\documentclass[12pt,twoside]{amsart}
\usepackage{amssymb}
\usepackage{verbatim}
\usepackage{amsmath}
\usepackage{bm}
\usepackage{a4wide}
\usepackage[latin1]{inputenc}
\usepackage[T1]{fontenc}
\usepackage{times}
\usepackage{amssymb,latexsym}
\usepackage{enumerate}
\usepackage{color}

\makeatletter
\newcommand{\sumprime}{\if@display\sideset{}{'}\sum%
            \else\sum'\fi}
\makeatother

\begin{document}

\numberwithin{equation}{section}

% define theorem environments
\newtheorem{theorem}{Theorem}[section]
\newtheorem{proposition}[theorem]{Proposition}
\newtheorem{conjecture}[theorem]{Conjecture}
\def\theconjecture{\unskip}
\newtheorem{corollary}[theorem]{Corollary}
\newtheorem{lemma}[theorem]{Lemma}
\newtheorem{observation}[theorem]{Observation}
\newtheorem{definition}{Definition}
\numberwithin{definition}{section} %\def\thedefinition{\unskip}
\newtheorem{remark}{Remark}
\def\theremark{\unskip}
\newtheorem{kl}{Key Lemma}
\def\thekl{\unskip}
\newtheorem{question}{Question}
\def\thequestion{\unskip}
\newtheorem{example}{Example}
\def\theexample{\unskip}
\newtheorem{problem}{Problem}

\thanks{Supported by National Natural Science Foundation of China, No. 12271101}

\address [Bo-Yong Chen] {Department of Mathematical Sciences, Fudan University, Shanghai, 200433, China}
\email{boychen@fudan.edu.cn}

\address [Yuanpu Xiong] {Department of Mathematical Sciences, Fudan University, Shanghai, 200433, China}
\email{ypxiong18@fudan.edu.cn}

\title{Some properties of the $p-$Bergman kernel and metric}
\author{Bo-Yong Chen and Yuanpu Xiong}

\date{}

\begin{abstract}
The $p-$Bergman kernel $K_p(\cdot)$ is shown to be of $C^{1,1/2}$ for $1<p<\infty$.   An unexpected relation between the off-diagonal $p-$Bergman kernel $K_p(\cdot,z)$ and certain weighted $L^2$ Bergman kernel is given for $1\le p\le 2$.  As applications,  we show that for each $1\le p\le 2$,  $K_p(\cdot,z)\in L^q(\Omega)$\/ for  $q< \frac{2pn}{2n-\alpha(\Omega)}$ and $|K_s(z)-K_p(z)| \lesssim |s-p||\log |s-p||$  whenever the hyperconvexity index $\alpha(\Omega)$ is positive. Counterexamples for $2<p<\infty$ are given respectively.   An optimal upper bound for the holomorphic sectional curvature of the $p-$Bergman metric when $2\le p<\infty$ is obtained.   For bounded $C^2$ domains,  it is shown that the Hardy space and the Bergman space satisfy $H^p(\Omega)\subset A^q(\Omega)$ where $q=p(1+\frac1n)$.  A new concept so-called the $p-$Schwarz content is introduced.  As  applications,  upper bounds of the Banach-Mazur distance between  $p-$Bergman spaces are given,  and $A^p(\Omega)$ is shown to be non-Chebyshev in $L^p(\Omega)$ for $0<p\le 1$.  For planar domains,  we obtain a rigidity theorem for the $p-$Bergman kernel (which is not valid in high dimensional cases),   and a characterization of non-isolated boundary points through completeness of the Narasimhan-Simha metric.  
\end{abstract}

\maketitle

\tableofcontents

\section{Introduction}

This paper is continuation of \cite{CZ},  where a general $p-$Bergman theory,  i.e.,  a theory for the $p-$Bergman kernel,   the $p-$Bergman metric and the $p-$Bergman space,  is developed.  We shall use the same notation and some basic properties in \cite{CZ} will be frequently used.   The $p-$Bergman space $A^p(\Omega)$ over a bounded domain $\Omega\subset \mathbb C^n$ is given by 
$$
A^p(\Omega)=\left\{f\in \mathcal O(\Omega): \|f\|_p^p:=\int_\Omega |f|^p<\infty\right\}
$$
(throughout this paper the integrals are with respect to the Lebesgue measure).
It is known  that the following minimizing problem
\begin{equation}\label{eq:MinProb}
 m_p(z):=\inf\left\{\|f\|_p:f\in A^p(\Omega),f(z)=1\right\}
 \end{equation}
admits at least one minimizer for $0<p<\infty$ and  exactly one minimizer $m_p(\cdot,z)$ for  $1\le p < \infty$.  The $p-$Bergman kernel and off-diagonal $p-$Bergman kernel are given by $K_p(z)=m_p(z)^{-p}$ for $0<p<\infty$ and  $K_p(z,w)=m_p(z,w)K_p(w)$ for $1\le p<\infty$ respectively. 
 Among the basic properties of the $p-$Bergman theory,  the most important is the following reproducing property:  
 \begin{equation}\label{eq:RPF}
 f(z) =  \int_\Omega |m_p(\cdot,z)|^{p-2}\,\overline{K_p(\cdot,z)}\,f,\ \ \ \forall\,f\in A^p(\Omega).
  \end{equation}
  
  A simple normal family argument yields that $m_p(\cdot)$ and $K_p(\cdot)$ are  Lipschitz continuous.  It is more difficult to verify whether they are differentiable.  The first main result of this paper is the following interior regularity theorem.  
  
  \begin{theorem}\label{th:C1,1/2}
 $K_p(\cdot)\in{C^{1,1/2}(\Omega)}$\/ for\/ $1<p<\infty$; moreover,  for any $1\le j\le 2n$ and $z\in \Omega$,
\begin{equation}\label{eq:partial derivative}
\frac{\partial{K_p}}{\partial{x}_j}(z)=p\mathrm{Re}\,\frac{\partial{K_p(\cdot,z)}}{\partial{x_j}}\bigg|_z,
\end{equation}
where $(x_1,\cdots,x_{2n})$ is the real coordinates in $\mathbb R^{2n}=\mathbb C^n$.
\end{theorem}

\begin{remark}
It is known from \cite{CZ} that there exist simply-connected pseudoconvex domains on which $K_p(\cdot)$ is not real-analytic whenever $p$ is an even larger than $4$.   It is of great interest to know the optimal regularity of $K_p(\cdot)$.   
\end{remark}

Next,  we present an unexpected relation between  $K_p(\cdot,\cdot)$ and certain weighted $L^2$ Bergman kernel for $1\le p\le 2$.  
Let $\varphi$ be a measurable real-valued function on $\Omega$ which enjoys the following property:  for each $z\in \Omega$ there exists a number $\alpha>0$ such that $e^{\alpha\varphi}$ is integrable in a neighborhood of $z$.   Set 
$$
A^2(\Omega,\varphi):=\left\{ f\in \mathcal O(\Omega): \int_\Omega |f|^2 e^{-\varphi}<\infty\right\}.
$$
According to Corollary 3.1 in \cite{PW},  $A^2(\Omega,\varphi)$ admits a unique Bergman reproducing kernel $K_{\Omega,\varphi}(\cdot,\cdot)$.
Given $p\ge 1$,  define
$$
A^2_{p,z}(\Omega):=A^2(\Omega,(2-p)\log |m_p(\cdot,z)|)\ \ \text{and}\ \ K_{2,p,z}(\cdot,\cdot):=K_{\Omega,(2-p)\log |m_p(\cdot,z)|}(\cdot,\cdot).  
$$
 
\begin{theorem}\label{th:p<2}
Let $1\le p\le 2$.  Then we have
\begin{equation}\label{eq:RF_0}
K_p(\cdot,z)=K_{2,p,z}(\cdot,z),\ \ \ \forall\,z\in \Omega.
\end{equation}
\end{theorem}

On the other hand,  we have

\begin{proposition}\label{prop:p>2}
For each $2<p<\infty$,  there exist a bounded  domain $\Omega\subset \mathbb C$ and a point $z\in \Omega$ such that
$$
K_p(\cdot,z)\neq K_{2,p,z}(\cdot,z).
$$
\end{proposition}

A related question\footnote{The authors are indebted to Professor Zbigniew B{\l}ocki for proposing this question during the Hayama Symposium 2022. }  is the following

\begin{problem}
Does there exist for each $2<p<\infty$ and $z\in \Omega$ a weight function $\varphi_z$ such that $K_p(\cdot,z)=K_{\Omega,\varphi_z}(\cdot,z)$?
\end{problem}

\begin{remark}
We claim that $\varphi_z$ cannot be chosen independent of $z$.  Otherwise,  $K_p(\cdot)$ would be real-analytic since $K_{\Omega,\varphi}(\cdot)$ is,   but this is not the case in general (cf. \cite{CZ}).
\end{remark}

In view of Theorem \ref{th:p<2},  we may apply the powerful $L^2$ method of the $\bar{\partial}-$equation to the $p-$Bergman kernel for $1\le p<2$.   For instance,  we are able to show  the following 

\begin{theorem}\label{th:Integ}
Let $1\le p\le 2$.  If\/ $\Omega$ is a bounded domain in $\mathbb C^n$ with the hyperconvexity index $\alpha(\Omega)>0$,  then the following properties hold:
\begin{enumerate}
\item[$(1)$]  $K_p(\cdot,z)\in L^q(\Omega)$\/ for any $z\in \Omega$ and $q< \frac{2pn}{2n-\alpha(\Omega)}$;
\item[$(2)$] $|K_s(z)-K_p(z)| \lesssim |s-p||\log |s-p||$ as $s\rightarrow p$.
 \end{enumerate}
\end{theorem}

Recall that the hyperconvexity index $\alpha(\Omega)$ of a domain $\Omega\subset \mathbb C^n$ is defined to be the supremum of those  $\alpha\ge 0$ such that there exists a negative continuous plurisubharmonic (psh) function $\rho$ on $\Omega$ with $-\rho\lesssim \delta^\alpha$,  where $\delta$ denotes the Euclidean boundary distance (cf.  \cite{Chen17}).   Note that 
the case $p=2$ in $(1)$ has been verified in \cite{Chen17}.

In contrast with Theorem \ref{th:Integ},  we have  

\begin{proposition}\label{prop:c-example}
  For each $2<p<\infty$,  there is a bounded domain $\Omega\subset \mathbb C$ with $\alpha(\Omega)>0$ which enjoys the following properties:
\begin{enumerate}  
\item[$(1)$]   For each $p'>p$ there exists a point $z\in \Omega$ such that $K_p(\cdot,z)\notin L^{p'}(\Omega)$;
 \item[$(2)$] There exists a point $z\in \Omega$ such that $K_s(z)$ is not continuous at $s=p$.
 \end{enumerate}
  \end{proposition}

Recall that the  $p-$Bergman metric is defined to be
$$
B_p(z;X):=  {K_p(z)^{-\frac1p}}\cdot {\sup}_f\ |Xf(z)|
$$
where the supremum is taken over all $f\in A^p(\Omega)$ with $f(z)=0$ and $\|f\|_p=1$.  Let  $\mathrm{HSC}_p(z;X)$ denote the holomorphic sectional curvature of $B_p(z;X)$.  It is a classical fact that  $\mathrm{HSC}_2(z;X)\le 2$ (see e.g.,  \cite{Kobayashi59}).  Use the calculus of variations,  we are able to show the following

\begin{theorem}\label{th:HSC_0}
For $2\le p<\infty$ we have
\begin{equation}\label{eq:HSC_0}
\mathrm{HSC}_p(z;X)\le \frac2p\cdot \frac{i\partial\bar{\partial} \log K_p(z;X)}{B_p(z;X)^2} +\frac{p}2,
\end{equation}
where $i\partial\bar{\partial} \log K_p(z;X)$ denotes the generalized Levi form of\/ $\log K_p(z)$. 
\end{theorem}

\begin{remark}
Note that  $i\partial\bar{\partial} \log K_p(z;X)\ge B_p(z;X)^2$ holds for $2\le p<\infty$ (cf.  \cite{CZ}). It remains open whether equality holds.
\end{remark}

The  Hardy space $H^p(\Omega)$ is closely  related to the Bergman space $A^p(\Omega)$.   The Hardy norm and the Bergman norm can be compared as follows. 

\begin{theorem}\label{th:HL_2}
Let $\Omega$ be  a bounded domain with $C^2$ boundary in $\mathbb C^n$ and $q=p(1+\frac1n)$.  Then 
 \begin{equation}\label{eq:HL_7}
\|f\|_{A^q(\Omega)} \lesssim    \|f\|_{H^p(\Omega)}, \ \ \ \forall\,f\in H^p(\Omega).
\end{equation}
In particular,  we have $S_p(z)^{1/p}\lesssim K_q(z)^{1/q}$,  where $S_p(z)$ denotes the $p-$Szeg\"o kernel given by
$$
S_p(z):= \sup\left\{|f(z)|^p: f\in H^p(\Omega),\,\|f\|_{H^p(\Omega)}\le 1\right\}.
$$
\end{theorem}

Inequality  \eqref{eq:HL_7} has a long history.  In 1920,  Carleman \cite{Carleman} first proved  \eqref{eq:HL_7} for $p=2$ with optimal constant in case $\Omega\subset\subset \mathbb C$ is  smooth and simply-connected,  which provides the first complex-analytic approach of the isoperimetric problem.  In 1932,  Hardy-Littlewood proved \eqref{eq:HL_7}   for the unit disc $\mathbb D$ (cf.  \cite{HL2},  Theorem 31).   
Extensions of these theorems can be found in Markovi\'c \cite{Markovic}  and the references therein.  

It is known from \cite{CZ} that $K_p(z)$ is an exhaustion function for $p<2(1+\frac1n)$ when $\Omega$ is a bounded $C^2$ pseudoconvex domain.  Theorem \ref{th:HL_2} together with Theorem 1.1 in \cite{ChenFu} yield

\begin{corollary}
If\/ $\Omega$ is a $\delta-$regular domain in $\mathbb C^n$,  then there exists a number $\alpha>1$ such that
$$
K_{2(1+\frac1n)}(z) \gtrsim \left(\delta(z)|\log \delta(z)|^{\alpha}\right)^{-(1+\frac1n)} 
$$
holds for all $z\in \Omega$ with $\delta(z)\le 1/2$.  In particular,  $K_{2(1+\frac1n)}(z)$ is an exhaustion function.
\end{corollary}   

Recall that a bounded $C^2$ domain $\Omega\subset\mathbb C^n$ is called a $\delta-$regular domain if there exist a bounded $C^2$ psh function $\lambda$ and a $C^2$ defining function $\rho$ on $\Omega$ such that $i\partial\bar{\partial}\lambda\ge \rho^{-1} i\partial\bar{\partial}\rho$.  
It is known from \cite{ChenFu} that this class of domains includes bounded domains with defining functions that are psh on $\partial \Omega$ and pseudoconvex domains  of D'Angelo finite type. 

The Banach-Mazur distance between two Banach spaces is of central importance in the study of local properties of Banach spaces  (see e.g.,   \cite{TJ},  Chapter 9). 

\begin{definition}
The Banach-Mazur distance between two Banach spaces $X,Y$ is defined to be
$$
d_{\mathrm{BM}}(X,Y):=\inf \left\{ \|T\|\cdot\|T^{-1}\|: T\in \mathcal L(X,Y) \right\}
$$
where $\mathcal L(X,Y)$ denotes the set of continuous isomorphisms from $X$ to $Y$.  When $X$ is not isomorphic to $Y$,  then we set $d_{\mathrm{BM}}(X,Y)=\infty$.   
\end{definition}

Since $p-$Bergman spaces  are Banach spaces for $1\le p <\infty$,  it is natural to consider the Banach-Mazur distance between them.  
Suppose $\Omega'\subset \Omega$ are bounded domains in $\mathbb C^n$ such that  every $f\in A^p(\Omega')$ extends to an element in $A^p(\Omega)$. This is satisfied for instance,  when $n\ge 2$ and $\Omega\backslash \Omega'$ is compact in $\Omega$,  in view of the Hartogs extension theorem.  Banach's open mapping theorem implies that the identity mapping $I:A^p(\Omega)\rightarrow A^p(\Omega')$ is an isomorphism,  so that   $d_{\mathrm BM}(A^p(\Omega'),A^p(\Omega))<\infty$.  A new question arises on how to estimate $d_{\mathrm BM}(A^p(\Omega'),A^p(\Omega))$ in terms of geometric conditions of $\Omega'$ and $\Omega$.  For this purpose,  we introduce the following concept,  which seems to be  of independent interest.  

\begin{definition}
Let $\Omega$ be a domain in $\mathbb C^n$ and $E$ a measurable subset of $\Omega$.  For $0<p < \infty$,  the $p-$Schwarz content of $E$ relative to $\Omega$ is defined to be
\begin{equation}\label{eq:p-SchwarzConst} 
s_p(E,\Omega):=\sup\left\{\frac{\int_E |f|^p}{\int_\Omega |f|^p}: f\in A^p(\Omega)\backslash \{0\}\right\}.
\end{equation}
\end{definition} 

\begin{remark}
One may define $s_\infty(E,\Omega):=\sup\{\|f\|_{L^\infty(E)}/\|f\|_{L^\infty(\Omega)}:f\in A^\infty(\Omega)\backslash \{0\}\}$,  which connects with the Schwarz lemma.  Moreover,   \eqref{eq:p-SchwarzConst} is somewhat similar to the definition of the classical relative capacity. 
But "capacity" is a well-known name,  so we use "content" instead.   
\end{remark}

\begin{proposition}\label{prop:BM}
Let $1\le p<\infty$.  Suppose $\Omega'\subset \Omega$ are bounded domains in $\mathbb C^n$ such that  every $f\in A^p(\Omega')$ extends to an element in $A^p(\Omega)$.  Then
\begin{equation}\label{eq:BM}
d_{\mathrm{BM}}(A^p(\Omega'),A^p(\Omega))\le \left(1-s_p(\Omega\backslash \Omega',\Omega)\right)^{-1/p}.
\end{equation}
\end{proposition}

We are left to find upper bounds for the $p-$Schwarz content.

\begin{proposition}\label{prop:p-Schwarz_0}
Let $\Omega$ be a bounded  domain in $\mathbb C^n$  and $E$ a measurable relatively compact subset in $\Omega$.  Set $d:=d(E,\partial \Omega)$.  Then the following properties hold:
\begin{enumerate}
\item[$(1)$] For $0<p<\infty$,  one has
$
s_p(E,\Omega) \le \frac{128/\lambda_1(\Omega)}{128/\lambda_1(\Omega)+d^2},
$
where   $\lambda_1(\Omega)$ is the first eigenvalue of the Laplacian on $\Omega$.
\item[$(2)$]  For  $1<p<\infty$,  one has
$
s_p(E,\Omega)\le \frac{C_{n,p}(\Omega)}{ C_{n,p}(\Omega)+ d^{p/ \tilde{p}}},
$
where  $\tilde{p}$ stands for the largest integer smaller than $p$.
\end{enumerate}
\end{proposition}

The $p-$Schwarz content has also an unexpected application to best approximations.  Recall that for a metric space $(X,d)$ and $Y\subset{X}$,  a point $y_0\in{Y}$ is said to be the best approximation of $x\in X$ if
$
d(x,y_0)=\inf_{y\in{Y}}d(x,y).
$
Let $\mathcal{P}_Y(x)$ denote the set of all  best approximations of $x$.  If $\mathcal{P}_Y(x)\neq\emptyset$ (respectively $\mathcal{P}_Y(x)$ is a singleton) for all $x\in{X}$, then we say that $Y$ is {\it proximal} (respectively {\it Chebyshev}) (cf. \cite{Singer}). It was shown in \cite{CZ} that $A^p(\Omega)$ is proximal in $L^p(\Omega)$ for $0<p\leq\infty$ and is Chebyshev in $L^p(\Omega)$ when $1<p<\infty$,  so that one may define the $p-$Bergman projection from $L^p(\Omega)$ to $A^p(\Omega)$ by the metric projection $f\mapsto \mathcal P_{A^p(\Omega)}(f)$ for $1<p<\infty$.  Using the $p-$Schwarz content in a crucial way,  we are able to show the following 

\begin{theorem}\label{th:not_Chebyshev}
If $\Omega$ is a bounded domain in $\mathbb{C}^n$ and $0<p\leq1$, then $A^p(\Omega)$ is not a Chebyshev set in $L^p(\Omega)$.  In particular,  the $p-$Bergman projection cannot be defined through the metric projection for $0<p\le 1$.  
\end{theorem}

Note that the  trivial estimate $K_p(z)\ge 1/|\Omega|$  always holds.  It is natural to ask 
 whether the lower bound can be achieved. 
 We give a lovely rigidity result as follows.  

\begin{theorem}\label{th:minimum_of_Bergman_kernel}
Let $\Omega$ be a bounded domain in $\mathbb{C}$.  Then the following properties hold:
\begin{enumerate}
\item[$(1)$] If $K_p(a)=1/|\Omega|$ for some $a\in\Omega$ and some $0<p<2$, then $\Omega$ is a disc centred at $a$.
\item[$(2)$] Suppose furthermore $|\overline{\Omega}^\circ\setminus\Omega|=0$.  
If $K_p(a)=1/|\Omega|$ for some $a\in\Omega$ and some $2\le p<\infty$,  then $\overline{\Omega}$ is a closed disc centred at $a$.
\end{enumerate}
\end{theorem}

Dong and Treuer have proved  that if $K_2(a)=1/|\Omega|$ for some $a\in \Omega$ then $\Omega$ is the complement of  some closed polar set  in a disc,  by using some nontrivial results connecting with the Suita conjecture (see \cite{DT}).  Our approach to Theorem \ref{th:minimum_of_Bergman_kernel}, which is in big part motivated by a rigidity theorem in terms of the mean-value property of harmonic functions due to  Epstein \cite{Epstein},  Epstein-Schiffer \cite{ES} and Kuran \cite{Kuran},  is rather elementary.   Note that  no analogous rigidity results as Theorem \ref{th:minimum_of_Bergman_kernel} exist in high dimensional cases,  for instance,  if $K_{\Omega,p}(a)=1/|\Omega|$ holds for some $a\in \Omega$,  then  $K_{\Omega\setminus E,p}(a)=1/|\Omega\setminus E|$ also holds when $E$ is a compact set in $\Omega\setminus \{a\}$ with $|E|=0$,  in view of the Hartogs extension theorem.   On the other hand,   it is easy to see that $K_p(0)=1/|\Omega|$ when $\Omega$ is a bounded complete circular domain in $\mathbb C^n$ (compare \S\,8).  
Thus it is natural  ask  whether a bounded fat domain $\Omega$ (i.e.,  $\overline{\Omega}^\circ=\Omega$)  has to be a complete circular domain if  $K_{\Omega,p}(0)=1/|\Omega|$ holds for some $0<p<\infty$.  Yet we have

\begin{proposition}\label{prop:Non-rigid}
There exists a bounded\/ {\rm non-circular}\/ domain $\Omega\subset \mathbb C^2$ with piecewise smooth boundary  such that $K_p(0)=1/|\Omega|$ for all $2\le p<\infty$.
\end{proposition}

The $p-$Bergman kernel can be used to produce various biholomorphic invariants. Narasimhan-Simha \cite{NS} introduced the following K\"ahler metric
$$
 ds^2_{p}:=\sum_{j,k=1}^n \frac{\partial^2 \log K_{2,p}(z)}{\partial z_j\partial\bar{z}_k} dz_j\otimes d\bar{z}_k,
 $$
where $K_{2,p}(z)=K_{\Omega,\log K_p}$.   It is known that $ds^2_p$ is invariant under biholomorphic mappings for simply-connected domains and for arbitrary bounded domains when $p=2/m$,  $m\in \mathbb Z^+$ (cf.  \cite{NS},  \cite{CZ}).   
 
\begin{theorem}\label{th:completeness_isolated}
 Let $\Omega$ be a bounded  domain in $\mathbb C$.  Then the following properties hold:
 \begin{enumerate}
\item[$(1)$] If $\Omega$ has no isolated boundary points, i.e., for every $w\in \partial \Omega$ there exists a sequence $\{w_j\}\subset \mathbb C\backslash \Omega$ such that $w_j\neq w$ for all $j$ and $w_j\rightarrow w$ as $j\rightarrow \infty$,  then $ds^2_p$ is complete for all $1\le p<2$.  
 \item
[$(2)$] If $\Omega$ admits an isolated boundary point, then $ds^2_p$ is  incomplete  for all $0<p<\infty$.
\end{enumerate}
\end{theorem} 

\begin{remark}
Thus $ds^2_1$ provides a conformal invariant which distinguishes between isolated boundary points and cluster boundary points in some sense.  
\end{remark}

We also present in  appendix a precise result on the asymptotic behavior of $K_p(z)$ for the punctured unit disc $\mathbb D^\ast$ as $z\rightarrow 0$.

\section{Differentiable properties of $K_p(z)$}
\subsection{Proof of Theorem \ref{th:C1,1/2}}
We  need the following lemma.

\begin{lemma}[Brezis-Lieb]\label{lm:Brezis Lieb}
Let $0<p<\infty$.  If\/ $f_k,f\in{L^p(\Omega)}$ satisfy $f_k\rightarrow{f}$ a.e.  and $\lVert{f_k}\rVert_p\rightarrow\lVert{f}\rVert_p$,  then $\lVert{f_k-f}\rVert_p\rightarrow0$.
\end{lemma}

\begin{proof}
Note that
\[|f_k-f|^p\leq{C_p(|f_k|^p+|f|^p)},\]
where $C_p:=\min\{1,2^{p-1}\}$. It suffices to apply Fatou's lemma to the function $C_p(|f_k|^p+|f|^p)-|f_k-f|^p$.
\end{proof}

Now we give a proof of Theorem \ref{th:C1,1/2} by using the duality method.
Let  $1< p<\infty$ be fixed.  Given $j$ and $z$,  the linear functionals 
\[
L_z(f):=f(z),\ \ \ L_{z,j}(f):=\frac{\partial{f}}{\partial{x_j}}(z),\ \ \ f\in{A^p(\Omega)}
\]
are continuous in view of the mean-value inequality and Cauchy's estimates.   The point is  
$$
\lVert{L_z}\rVert = \sup\{|f(z)|: f\in A^p(\Omega),\,\|f\|_p \le 1\} =K_p(z)^{1/p},
$$
 which motivates us to use functional analysis.

Let $e_1,\cdots,e_{2n}$ be the standard basis in $\mathbb{R}^{2n}=\mathbb{C}^n$.   Consider the following linear functionals
\[
\Lambda_t:=\frac{L_{z+te_j}-L_z}{t}-L_{z,j},\ \ \ t\in\mathbb{C}\setminus\{0\}.
\]
Since for every $ f\in{A^p(\Omega)}$,
\[
|\Lambda_t(f)|=\left|\frac{f(z+te_j)-f(z)}{t}-\frac{\partial{f}}{\partial{x_j}}(z)\right|=O(|t|)\ \ \ (t\rightarrow 0),
\]
we have $\lVert{\Lambda_t}\rVert=O(|t|)$ in view of 
 the Banach-Steinhaus theorem,  so that
\begin{equation}\label{eq:differentiable I}
\lVert{L_{z+te_j}-L_z-tL_{z,j}}\rVert=O(|t|^2).
\end{equation}
Set
\begin{equation}\label{eq:gz}
g_z:=\frac{|m_p(\cdot,z)|^{p-2}m_p(\cdot,z)}{m_p(z)^p}.
\end{equation}
Since $\|m_p(\cdot,z)\|_p=m_p(z)$,  we have $g_z\in{L^q(\Omega)}$ and 
\begin{equation}\label{eq:gz_norm}
\lVert{g_z}\rVert_q=m_p(z)^{-1}=K_p(z)^{1/p}=\lVert{L_z}\rVert,
\end{equation}
where $1/p+1/q=1$. Moreover, it follows from \eqref{eq:RPF} that
\begin{equation}\label{eq:reproducing I}
L_z(f)=f(z)=\int_\Omega{f}\overline{g}_z,\ \ \ \forall\, f\in{A^p(\Omega)}.
\end{equation}
By the Hahn-Banach theorem,  $L_{z,j}$ can be extended to a continuous linear functional $\widetilde{L}_{z,j}\in{(L^p(\Omega))^*}$ with $\lVert{L_{z,j}}\rVert=\lVert{\widetilde{L}_{z,j}}\rVert$. Let $g_{z,j}\in{L^q(\Omega)}$ be the unique representation of $\widetilde{L}_{z,j}$, so that $\lVert{g_{z,j}}\rVert_q=\lVert{\widetilde{L}_{z,j}}\rVert$ and
\begin{equation}\label{eq:reproducing II}
L_{z,j}(f)=\widetilde{L}_{z,j}(f)=\int_\Omega{f}\overline{g}_{z,j},\ \ \ \forall\, f\in{A^p(\Omega)}.
\end{equation}

From now on we  assume that $t\in\mathbb{R}$.  \eqref{eq:reproducing I} together with \eqref{eq:reproducing II} yield
\[
(L_z+tL_{z,j})(f)=\int_\Omega{f}\,\overline{g_z+tg_{z,j}},\ \ \ \forall\,{f\in{A^p(\Omega)}},
\]
so that
\begin{equation}\label{eq:L g}
\lVert{L_z+tL_{z,j}}\rVert\leq\lVert{g_z+tg_{z,j}}\rVert_q.
\end{equation}
Analogously,  we have
\begin{equation}\label{eq:L g 2}
\lVert{L_{z+te_j}-tL_{z,j}}\rVert\leq\lVert{g_{z+te_j}-tg_{z,j}}\rVert_q.
\end{equation}
Use \eqref{eq:differentiable I}, \eqref{eq:L g} and \eqref{eq:L g 2}, we obtain
\begin{eqnarray}\label{eq:difference I}
K_p(z+te_j)^{q/p}-K_p(z)^{q/p} &=& \lVert{L_{z+te_j}}\rVert^q-\lVert{L_z}\rVert^q\nonumber\\
&=& \lVert{L_z+tL_{z,j}}\rVert^q-\lVert{L_z}\rVert^q+O(|t|^2)\nonumber\\
&\leq& \lVert{g_z+tg_{z,j}}\rVert_q^q-\lVert{g_z}\rVert_q^q+O(|t|^2)
\end{eqnarray}
and
\begin{eqnarray}\label{eq:difference II}
K_p(z+te_j)^{q/p}-K_p(z)^{q/p} &=& \lVert{L_{z+te_j}}\rVert^q-\lVert{L_z}\rVert^q\nonumber\\
&=& \lVert{L_{z+te_j}}\rVert^q-\lVert{L_{z+te_j}-tL_{z,j}}\rVert^q-O(|t|^2)\nonumber\\
&\geq& \lVert{g_{z+te_j}}\rVert_q^q-\lVert{g_{z+te_j}-tg_{z,j}}\rVert_q^q-O(|t|^2).
\end{eqnarray}
Next,  let us  recall the following inequalities (cf. \cite{CZ}, (4.3) and (4.4)):
\begin{eqnarray*}
|b|^q &\geq& |a|^q+q\mathrm{Re}\,\left\{|a|^{q-2}\overline{a}(b-a)\right\}+A_q|b-a|^2(|a|+|b|)^{q-2},\ \ \ (1<q\leq2),\\
|b|^q &\geq& |a|^q+q\mathrm{Re}\,\left\{|a|^{q-2}\overline{a}(b-a)\right\}+\frac{1}{4^{q+3}}|b-a|^q,\ \ \ (q>2),
\end{eqnarray*}
where $a,b\in\mathbb{C}$ and $A_q:=\frac{q}{2}\min\{1,q-1\}$.  Choose $a=g_z+tg_{z,j}$ and $b=g_z$ in both cases and taking integration over $\Omega$, we obtain
\begin{eqnarray}\label{eq:difference I1}
\lVert{g_z+tg_{z,j}}\rVert_q^q-\lVert{g_z}\rVert_q^q &\leq& qt\,\mathrm{Re}\,\int_\Omega|g_z+tg_{z,j}|^{q-2}\overline{(g_z+tg_{z,j})}g_{z,j}\nonumber\\
&=:& qt\,\mathrm{Re}\,\int_\Omega|g_z|^{q-2}\overline{g}_zg_{z,j}+qt\omega_1(t),
\end{eqnarray}
where
\[
\omega_1(t):=\mathrm{Re}\,\int_\Omega\left\{|g_z+tg_{z,j}|^{q-2}\overline{(g_z+tg_{z,j})}-|g_z|^{q-2}\overline{g}_z\right\}g_{z,j}.
\]
Since
\[
\left||g_z+tg_{z,j}|^{q-2}\overline{(g_z+tg_{z,j})}g_{z,j}\right|=|g_z+tg_{z,j}|^{q-1}|g_{z,j}|\leq(|g_z|+|g_{z,j}|)^{q-1}|g_{z,j}|\in L^1(\Omega)
\]
in view of H\"older's inequality 
when $|t|<1$, we infer from the dominated convergence theorem that $\omega_1(t)\rightarrow0$ as $t\rightarrow0$. Hence \eqref{eq:difference I} combined with \eqref{eq:difference I1} gives
\begin{equation}\label{eq:differential formula I}
K_p(z+te_j)^{q/p}-K_p(z)^{q/p}\leq{qt}\,\mathrm{Re}\,\int_\Omega|g_z|^{q-2}\overline{g}_zg_{z,j}+o(t).
\end{equation}
Analogously, we have
\begin{eqnarray}\label{eq:difference II1}
\lVert{g_{z+te_j}}\rVert_q^q-\lVert{g_{z+te_j}-tg_{z,j}}\rVert_q^q &\geq& qt\,\mathrm{Re}\,\int_\Omega|g_{z+te_j}-tg_{z,j}|^{q-2}\overline{(g_{z+te_j}-tg_{z,j})}g_{z,j}\nonumber\\
&=& qt\,\mathrm{Re}\,\int_\Omega|g_z|^{q-2}\overline{g}_zg_{z,j}+qt\omega_2(t),
\end{eqnarray}
where
\[
\omega_2(t):=\mathrm{Re}\,\int_\Omega\left\{|g_{z+te_j}-tg_{z,j}|^{q-2}\overline{(g_{z+te_j}-tg_{z,j})}-|g_z|^{q-2}\overline{g}_z\right\}g_{z,j}.
\]
For the sake of convenience, we set
\[
h_t:=|g_{z+te_j}-tg_{z,j}|^{q-2}\overline{(g_{z+te_j}-tg_{z,j})},\ \ \ h:=|g_z|^{q-2}\overline{g}_z.
\]
Since  ${m_p(\zeta,z)}$ is continuous in $(\zeta,z)$ (cf. \cite{CZ}, Proposition 2.11),  we see that $h_t\rightarrow{h}$ pointwisely as $t\rightarrow0$.  Moreover,  since
\begin{eqnarray*}
\lVert{h_t}\rVert_p &=& \lVert{g_{z+te_j}-tg_{z,j}}\rVert^{q/p}_q=\lVert{g_{z+te_j}}\rVert^{q/p}_q+O(|t|)\\
&=& K_p(z+te_j)^{q/p^2}+O(|t|)\ \ \ \ \ (\text{by\ }\eqref{eq:gz_norm})\\
&\rightarrow & K_p(z)^{q/p^2}=\lVert{h}\rVert_p\ \ \ (t\rightarrow0),
\end{eqnarray*}
we infer from Lemma \ref{lm:Brezis Lieb} that
\[
|\omega_2(t)|\leq\lVert{h_t-h}\rVert_p\lVert{g_{z,j}}\rVert_q\rightarrow0.
\]
This together with \eqref{eq:difference II} and \eqref{eq:difference II1} give
\begin{equation}\label{eq:differential formula II}
K_p(z+te_j)^{q/p}-K_p(z)^{q/p}\geq{qt}\mathrm{Re}\,\int_\Omega|g_z|^{q-2}\overline{g}_zg_{z,j}+o(t).
\end{equation}
By \eqref{eq:differential formula I} and \eqref{eq:differential formula II}, we conclude that 
\begin{equation}\label{eq:differential formula directional I}
\frac{\partial{K_p^{q/p}}}{\partial{x_j}}(z)=q\,\mathrm{Re}\,\int_\Omega|g_z|^{q-2}\overline{g}_zg_{z,j}.
\end{equation}
In particular, the partial derivative $\partial{K_p}/\partial{x_j}$ exists and
\begin{equation}\label{eq:differential formula directional II}
\frac{\partial{K_p}}{\partial{x_j}}(z)=\frac{p}{q}K_p(z)^{1-q/p}\frac{\partial{K_p^{q/p}}}{\partial{x_j}}(z)=\frac{p}{q}m_p(z)^{q-p}\frac{\partial{K_p^{q/p}}}{\partial{x_j}}(z).
\end{equation}
Substitute \eqref{eq:gz}, \eqref{eq:reproducing II} and \eqref{eq:differential formula directional I} into \eqref{eq:differential formula directional II},  we obtain
\begin{eqnarray*}
\frac{\partial{K_p}}{\partial{x_j}}(z) &=& pm_p(z)^{q-p}\,\mathrm{Re}\,\int_\Omega\frac{|m_p(\cdot,z)|^{(p-1)(q-2)+p-2}}{m_p(z)^{p(q-1)}}\overline{m_p(\cdot,z)}g_{z,j}\\
&=& \frac{p}{m_p(z)^p}\mathrm{Re}\,\int_\Omega\overline{m_p(\cdot,z)}g_{z,j}\\
&=& \frac{p}{m_p(z)^p}\mathrm{Re}\,L_{z,j}\big(m_p(\cdot,z)\big)\\
&=& p\mathrm{Re}\,\frac{\partial{K_p(\cdot,z)}}{\partial{x}_j}\bigg|_z,
\end{eqnarray*}
for  $(p-1)(q-2)+p-2=0$ and $p(q-1)=q$.

It remains to verify that
\begin{equation}\label{eq:C1/2}
z\mapsto\partial{K_p(\cdot,z)}/\partial{x_j}|_z\in{C^{1/2}(\Omega)}.
\end{equation}
By Theorem 4.7 in \cite{CZ}, we know that for given domains $\Omega'\subset\subset\Omega$,
\[
|K_p(w,z)-K_p(w,z')|\leq{C|z-z'|^{1/2}},\ \ \ w,z,z'\in\Omega'
\]
for some constant $C=C(p,\Omega',\Omega)>0$.  Given two points $z,z'\in\Omega'$,  the function
\[
\widehat{h}:=\frac{K_p(\cdot,z)-K_p(\cdot,z')}{|z-z'|^{1/2}}
\]
is holomorphic on $\Omega$ and satisfies $\sup_{\Omega'}|\widehat h|\leq{C}$.
Thus Cauchy's estimate yields that for each $z'\in \Omega''\subset \subset \Omega'$,  
\[\left|\frac{\partial{\widehat h}}{\partial{x_j}}(z')\right|\leq{MC},\]
where the constant $M$ depends only on $\Omega'$, $\Omega''$ and $\Omega$. On the other hand, since $\partial{K_p(\cdot,z)}/\partial{x_j}$ is a holomorphic function on $\Omega$, we have
\[
\Bigg|\frac{\partial{K_p(\cdot,z)}}{\partial{x_j}}\bigg|_z-\frac{\partial{K_p(\cdot,z)}}{\partial{x_j}}\bigg|_{z'}\Bigg|\leq{C'|z-z'|},
\]
where $C'=C'(\Omega'',\Omega)>0$. Hence
\begin{eqnarray*}
\Bigg|\frac{\partial{K_p(\cdot,z)}}{\partial{x_j}}\bigg|_z-\frac{\partial{K_p(\cdot,z')}}{\partial{x_j}}\bigg|_{z'}\Bigg| &\leq& \Bigg|\frac{\partial{K_p(\cdot,z)}}{\partial{x_j}}\bigg|_z-\frac{\partial{K_p(\cdot,z)}}{\partial{x_j}}\bigg|_{z'}\Bigg|+\left|\frac{\partial{\widehat h}}{\partial{x_j}}(z')\right||z-z'|^{1/2}\\
&\leq&{C'|z-z'|+MC|z-z'|^{1/2}},
\end{eqnarray*}
i.e., \eqref{eq:C1/2} holds.

\subsection{An application of Theorem \ref{th:C1,1/2}}
 Recall that the\/ {\it generalized complex Laplacian}  of an upper semicontinuous function $u$ on a domain in $\mathbb C$ is given by
$$
\Box u(t):=\liminf_{r\rightarrow 0+}\frac1{r^2}\left\{\frac1{2\pi}\int_0^{2\pi}u(t+re^{i\theta})d\theta-u(t)\right\}.
$$
If $u$ is an upper semicontinuous function on a domain in $\mathbb C^n$,  then we define the\/ {\it generalized Levi form} of $u$ to be 
$$
i\partial\bar\partial u(z;X):= \left.\Box u(z+tX)\right|_{t=0},
$$   
where we identify $X=\sum X_j \partial/\partial z_j$ with $(X_1,\cdots,X_n)$.

\begin{proposition}
Let $1<p<\infty$.  Let $F:\Omega_1\rightarrow \Omega_2$ be a biholomorphic mapping between bounded simply-connected domains.   Then
$$
i\partial\bar\partial \log K_{\Omega_1,p}(z;X)= i\partial\bar\partial \log K_{\Omega_2,p}(F(z);F_\ast X).
$$
\end{proposition}

Given $0<p<\infty$,  define
$$
\widehat{B}_p(z;X)=\widehat{B}_{\Omega,p}(z;X):=\sup_\sigma \left\{ \Box \log K_p\circ \sigma (0) \right\},
$$
where $\sigma$ is taken over all holomorphic mappings from some disc $\mathbb D_r:=\{z\in \mathbb C: |z|<r\}$ to $\Omega$ such that $\sigma(0)=z$ and $\sigma'(0)=X$.
Clearly,  it suffices to verify the following two lemmas.  

\begin{lemma}\label{lm:Invariant_1}
For $1<p<\infty$ we have
$$
i\partial\bar\partial \log K_{p}(z;X)=\widehat{B}_p(z;X).
$$
\end{lemma}

\begin{lemma}\label{lm:Invariant_2}
Let $0<p<\infty$.  Let $F:\Omega_1\rightarrow \Omega_2$ be a biholomorphic mapping between bounded simply-connected domains.   Then
\begin{equation}\label{eq:invariant_2}
\widehat{B}_{\Omega_1,p}(z;X) = \widehat{B}_{\Omega_2,p}(F(z);F_\ast X).
\end{equation}
Moreover, \eqref{eq:invariant_2} holds for arbitrary bounded domains whenever $2/p\in \mathbb Z^+$.
\end{lemma}

\begin{proof}[Proof of Lemma \ref{lm:Invariant_1}]
Fix a holomorphic mapping $\sigma:\mathbb D_r\rightarrow \Omega$ with $\sigma(0)=z$ and $\sigma'(0)=X$.  Set
$$
\tau(t):=\sigma(t)-(\sigma(0)+t\sigma'(0))=\sigma(t)-(z+tX).
$$
Clearly,  $\tau$ is a holomorphic mapping from $\mathbb D_r$ to $\Omega$ such that $|\tau(t)|=O(|t|^2)$.  Since $K_p(\cdot)$ is $C^1$ in view of Theorem \ref{th:C1,1/2},  we have
\begin{eqnarray*}
 \log K_p\circ \sigma(r e^{i\theta}) & = &  \log K_p( z+ r e^{i\theta}X + \tau(r e^{i\theta}))\\
 & = &  \log K_p( z+ r e^{i\theta}X) + \sum_{j=1}^n \frac{\partial \log K_p}{\partial \zeta_j}( z+ r e^{i\theta}X)\,\tau_j(re^{i\theta})\\
 && + \sum_{j=1}^n \frac{\partial \log K_p}{\partial \bar\zeta_j}( z+ r e^{i\theta}X)\,\overline{\tau_j(re^{i\theta})} + o(r^2)\\
 & = &  \log K_p( z+ r e^{i\theta}X) + \sum_{j=1}^n \frac{\partial \log K_p}{\partial \zeta_j}( z)\,\tau_j(re^{i\theta})\\
 && + \sum_{j=1}^n \frac{\partial \log K_p}{\partial \bar\zeta_j}( z)\,\overline{\tau_j(re^{i\theta})} + o(r^2)
\end{eqnarray*} 
since $\int_0^{2\pi} \tau_j(re^{i\theta})d\theta=\tau_j(0)=0$.
It follows that 
$$
\Box \log K_p(\sigma)(0) = i\partial\bar\partial \log K_p(z;X).
$$
Since $\sigma$ can be arbitrarily chosen,  we are done.
\end{proof}

\begin{proof}[Proof of Lemma \ref{lm:Invariant_2}]
Given a holomorphic mapping $\sigma_1:\mathbb D_r\rightarrow \Omega_1$ with $\sigma(0)=z$ and $\sigma_1'(0)=X$,  define $\sigma_2:=F\circ \sigma_1$.  Clearly,  $\sigma_2$ is a holomorphic mapping from $\mathbb D_r$ to $\Omega_2$ satisfying $\sigma_2(0)=F(z)$ and $\sigma_2'(0)=F_\ast X$.  By the transformation rule of $K_p$ (cf.  \cite{CZ},  Proposition 2.7),  we have
$$
\log K_{\Omega_1,p}\circ \sigma_1 = \log K_{\Omega_2,p}\circ \sigma_2 + \log |J_F|^2.
$$
Since $\log |J_F|^2$ is pluriharmonic,  we infer from the mean-value property that 
\begin{eqnarray*}
&& \frac1{2\pi}\int_0^{2\pi} \log K_{\Omega_1,p}\circ \sigma_1(re^{i\theta}) d\theta - \log K_{\Omega_1,p}\circ \sigma_1(0)\\
& = &  \frac1{2\pi}\int_0^{2\pi} \log K_{\Omega_2,p}\circ \sigma_2(re^{i\theta}) d\theta - \log K_{\Omega_2,p}\circ \sigma_2(0),
\end{eqnarray*}
from which the assertion immediately follows.
\end{proof}

\section{The case $1\le p \le 2$}
\subsection{Proof of Theorem \ref{th:p<2}}
We first verify that $A^2_{p,z}(\Omega)\subset A^p(\Omega)$ for $1\le p\le 2$.  To see this, simply note that for any $f\in A^2_{p,z}(\Omega)$,
\begin{eqnarray*}
\int_\Omega |f|^p & = & \int_\Omega \left(|m_p(\cdot,z)|^{\frac{p}2(p-2)}|f|^p\right)|m_p(\cdot,z)|^{\frac{p}2(2-p)}\\
& \le & \left(\int_\Omega |m_p(\cdot,z)|^{p-2} |f|^2\right)^{p/2} \left(\int_\Omega |m_p(\cdot,z)|^p \right)^{1-p/2}\\
& = & m_p(z)^{p(1-p/2)} \left(\int_\Omega |m_p(\cdot,z)|^{p-2} |f|^2\right)^{p/2}\\
& < & \infty.
\end{eqnarray*}
Thus  we have
\begin{equation}\label{eq:RF_1}
f(z)= \int_\Omega |m_p(\cdot,z)|^{p-2}\overline{K_p(\cdot,z)}f,\ \ \ \forall\,f\in A^2_{p,z}(\Omega),
\end{equation}
in view of \eqref{eq:RPF}.
On the other hand,  the reproducing formula for $A^2_{p,z}(\Omega)$ gives
\begin{equation}\label{eq:RF_2}
f(z)= \int_\Omega |m_p(\cdot,z)|^{p-2}\overline{K_{2,p,z}(\cdot,z)}f,\ \ \ \forall\,f\in A^2_{p,z}(\Omega).
\end{equation}
Thus
\begin{equation}\label{eq:RF_3}
 \int_\Omega |m_p(\cdot,z)|^{p-2}\overline{(K_p(\cdot,z)-K_{2,p,z}(\cdot,z))}f=0,\ \ \ \forall\,f\in A^2_{p,z}(\Omega).
\end{equation}
Since 
\begin{equation}\label{eq:RF_4}
 \int_\Omega |m_p(\cdot,z)|^{p-2}|K_p(\cdot,z)|^2=K_p(z)^2 \int_\Omega |m_p(\cdot,z)|^p=
 K_p(z),  
\end{equation}
we see that $K_p(\cdot,z)\in A^2_{p,z}(\Omega)$.   Substitute $f:=K_p(\cdot,z)-K_{2,p,z}(\cdot,z)$ into \eqref{eq:RF_3},  we immediately get $f=0$,  i.e.,  \eqref{eq:RF_0} holds.

\subsection{Proof of Theorem \ref{th:Integ}}
Let $A^2(\Omega,\varphi)$ and  $K_{\Omega,\varphi}$ be given as \S\,1.  We need the following $L^2$ boundary decay estimate for $K_{\Omega,\varphi}(\cdot,z)$.

\begin{proposition}\label{prop:BergmanIntegral}
Let $\Omega\subset {\mathbb C}^n$ be a pseudoconvex domain and $\rho$ a negative continuous psh  function on $\Omega$.
 Set
  $$
 \Omega_t=\{\zeta\in \Omega:-\rho(\zeta)>t\},\ \ \ t>0.
 $$
   Let $a>0$ be given.
  For every $0<r<1$, there exist constants $\varepsilon_r,C_r>0$ such that
  \begin{equation}\label{eq:2.1}
  \int_{-\rho\le \varepsilon} |K_{\Omega,\varphi}(\cdot,z)|^2e^{-\varphi} \le C_{r}\, K_{\Omega_a,\varphi}(z) (\varepsilon/a)^{r}
  \end{equation}
  for all $z\in \Omega_a$ and $\varepsilon\le \varepsilon_r a$.
  \end{proposition} 
  
  Proposition \ref{prop:BergmanIntegral} has been verified in \cite{Chen17} for the case $\varphi=0$.  Although the general case is  the same,   we still include a proof here for the sake of completeness.
  
  \begin{proof}
  Let $P_{\Omega,\varphi}: L^2(\Omega,\varphi)\rightarrow A^2(\Omega,\varphi)$ be the Bergman projection.   Then we have
  \begin{equation}\label{eq:BCh}
  \int_\Omega |P_{\Omega,\varphi}(f)|^2 (-\rho)^{-r}e^{-\varphi} \le C_r \int_\Omega |f|^2 (-\rho)^{-r}e^{-\varphi}
  \end{equation}
  for any $0<r<1$ and $f\in L^2(\Omega,\varphi)$ (cf. \cite{BCh}).  Let $\chi_E$ denote the characteristic function of a set $E$.  
  For $f:=\chi_{\Omega_a} K_{\Omega_a,\varphi}(\cdot,z)$,  we have  
\begin{eqnarray}\label{eq:BCh_2}
P_{\Omega,\varphi}(f)(\zeta) & = & \int_\Omega \chi_{\Omega_a}(\cdot) K_{\Omega_a,\varphi}(\cdot,z)\overline{K_{\Omega,\varphi}(\cdot,\zeta)}\nonumber\\
& = & \overline{\int_{\Omega_a} K_{\Omega,\varphi}(\cdot,\zeta) \overline{K_{\Omega_a,\varphi}(\cdot,z)}}\nonumber\\
& = &  K_{\Omega,\varphi}(\zeta,z)
\end{eqnarray} 
in view of the reproducing property.  
   By \eqref{eq:BCh} and \eqref{eq:BCh_2},  we obtain
    \begin{eqnarray*}
   \int_\Omega |K_{\Omega,\varphi}(\cdot,z)|^2 (-\rho)^{-r} e^{-\varphi} & \le & C_r \int_{\Omega_a} | K_{\Omega_a,\varphi}(\cdot,z) |^2 (-\rho)^{-r} e^{-\varphi}\\
   & \le & C_r a^{-r} K_{\Omega_a,\varphi}(z),
    \end{eqnarray*}
    from which the assertion immediately follows,  for 
    $$
    \int_\Omega |K_{\Omega,\varphi}(\cdot,z)|^2 (-\rho)^{-r} e^{-\varphi} \ge \varepsilon^{-r}\int_{-\rho\le \varepsilon} |K_{\Omega,\varphi}(\cdot,z)|^2 e^{-\varphi}.
    $$
  \end{proof}
  
  \begin{proof}[Proof of Theorem \ref{th:Integ}]
$(1)$  For any $0<\alpha<\alpha(\Omega)$,  we take a continuous negative plurisubharmonic function $\rho$ on $\Omega$ such that $-\rho\le C_\alpha \delta^\alpha$ for suitable constant $C_\alpha>0$.  By Theorem \ref{th:p<2},  we have $K_{2,p,z}(\cdot,z)=K_p(z)m_p(\cdot,z)$.   Apply Proposition \ref{prop:BergmanIntegral} with   $\varphi:=(2-p)\log |m_p(\cdot,z)|$ and $a=-\rho(z)/2$,  we obtain for $\varepsilon\ll1$,
  \begin{eqnarray*}
  \int_{-\rho\le \varepsilon} |m_p(\cdot,z)|^p  & = & K_p(z)^{-2} \int_{-\rho\le \varepsilon} |m_p(\cdot,z)|^{p-2} |K_{2,p,z}(\cdot,z)|^2\\
  & \le & C_{r} \frac{K_{\Omega_{a},\varphi}(z)}{K_p(z)^2}\left(\frac{\varepsilon}{-\rho(z)}\right)^r.
  \end{eqnarray*}
Define $d(z)=d(z,\partial \Omega_a)$.   Note that the mean-value  inequality yields 
$$
1/|\Omega| \le K_p(z)\le C_n \delta(z)^{-2n},
$$
  $$
\sup_{B(z,d(z))} |m_p(\cdot,z)|^p \le C_n d(z)^{-2n} m_p(z)^p\le C_n d(z)^{-2n} |\Omega|,
$$
and 
$$
K_{\Omega_{a},\varphi}(z)\le C_n d(z)^{-2n} \sup_{B(z,d(z))} e^{\varphi}=C_n d(z)^{-2n} \left(\sup_{B(z,d(z))} |m_p(\cdot,z)|^p\right)^{\frac2p-1}.
$$  
As $1\le p\le 2$,  we obtain
  \begin{equation}\label{eq:B-Integ_1}
  \int_{\delta\le \varepsilon} |m_p(\cdot,z)|^p\lesssim \varepsilon^{r\alpha}
  \end{equation}
  where the implicit constant depends only on $n,r,\alpha,z$.  
  This combined with the mean-value inequality gives
  $$
  |m_p(\zeta,z)|^p \le C_n \delta(\zeta)^{-2n} \int_{\delta\le 2\delta(\zeta)} |m_p(\cdot,z)|^p\lesssim \delta(\zeta)^{r\alpha-2n},\ \ \ \forall\, \zeta\in\Omega.
  $$
  Let $\tau>0$.  Then we have
  \begin{eqnarray*}
  \int_\Omega |m_p(\cdot,z)|^{p+\tau} & \lesssim & 1+ \sum_{k=1}^\infty \int_{2^{-k-1}<\delta \le 2^{-k}} |m_p(\cdot,z)|^{p+\tau}\\
  & \lesssim & 1+ \sum_{k=1}^\infty 2^{k\tau(2n-r\alpha)/p} \int_{\delta \le 2^{-k}} |m_p(\cdot,z)|^{p}\\
  & \lesssim & 1+ \sum_{k=1}^\infty 2^{k\tau(2n-r\alpha)/p-kr\alpha}\\
  & < & \infty
  \end{eqnarray*}
  provided $\tau(2n-r\alpha)/p<r\alpha$,  i.e.,  $\tau<\frac{pr\alpha}{2n-r\alpha}$.   Since $r$ and $\alpha$ can be arbitrarily close to $1$ and $\alpha(\Omega)$ respectively,  we are done.

  $(2)$ Without loss of generality,  we assume $s>p$.  Since $|\Omega|^{\frac1t}\cdot K_t(z)^{\frac1t}$ is  nonincreasing in $t$ (cf.  (6.3) in \cite{CZ}),  we see that 
  $$
  K_p(z)\ge |\Omega|^{\frac{p}s-1}\cdot K_s(z)^{\frac{p}s} = K_s(z)\left( |\Omega|K_s(z)\right)^{\frac{p}s-1} \ge K_s(z)- C(s-p)
  $$
  when $s$ is sufficiently close to $p$,
  where $C$ is  a suitable constant depending only on $z,\Omega$. On the other hand,  
 we infer from \eqref{eq:B-Integ_1}  that
\begin{eqnarray*}
\int_{\delta\le \varepsilon} |K_p(\cdot,z)|^{s} & \le &  \sum_{k=0}^\infty \int_{2^{-k-1}\varepsilon <\delta \le 2^{-k}\varepsilon} |K_p(\cdot,z)|^{p+s-p}\\
  & \lesssim &  \sum_{k=0}^\infty (2^{-k}\varepsilon)^{-(s-p)(2n-r\alpha)/p} \int_{\delta \le 2^{-k}\varepsilon} |K_p(\cdot,z)|^{p}\\
  & \lesssim &  \sum_{k=0}^\infty (2^{-k}\varepsilon)^{-(s-p)(2n-r\alpha)/p+r\alpha}\\
  & \lesssim & \varepsilon^{\alpha(\Omega)/2}
\end{eqnarray*} 
provided that $1-r$ and $s-p$ are sufficiently small.   
  Thus
  \begin{eqnarray*}
 \int_\Omega |K_p(\cdot,z)|^s & = &  \int_{\delta> \varepsilon} |K_p(\cdot,z)|^p |K_p(\cdot,z)|^{s-p} + \int_{\delta\le \varepsilon} |K_p(\cdot,z)|^s\\
 & \le & (C_1\varepsilon^{-\frac{2n}p})^{s-p} \int_\Omega |K_p(\cdot,z)|^p + C_2 \varepsilon^{\alpha(\Omega)/2}.
  \end{eqnarray*}
  Now we take $\varepsilon=(s-p)^{2/\alpha(\Omega)}$.  Since
  $$
  (C_1\varepsilon^{-\frac{2n}p})^{s-p}\le 1+ C_3 (s-p)|\log(s-p)|\ \ \ \text{and}\ \ \ \varepsilon^{\alpha(\Omega)/2}=s-p,
  $$
  it follows that
  \begin{eqnarray*}
  \int_\Omega |K_p(\cdot,z)|^s & \le & \int_\Omega |K_p(\cdot,z)|^p + C_4 (s-p)|\log (s-p)|\\
  & = & K_p(z)^{p-1} + C_4 (s-p)|\log (s-p)|.
  \end{eqnarray*}
  Thus 
  \begin{eqnarray*}
  K_s(z)\ge \frac{K_p(z)^s}{\int_\Omega |K_p(\cdot,z)|^s} & \ge & \frac{K_p(z)^p -C_5(s-p)}{K_p(z)^{p-1}+C_4(s-p)|\log(s-p)|}\\
  &\ge&  K_p(z)-C_6(s-p)|\log(s-p)|.
  \end{eqnarray*}
  \end{proof}

 \section{The case $2<p<\infty$}
 Let us first recall the following result on removable singularities.
  
  \begin{theorem}[cf.  Hedberg \cite{Hedberg}]\label{th:Hedberg}
 Let $D$ be a domain in $\mathbb C$ and $E$ a compact set in $D$.  Let $1<p<\infty$.  Then $A^p(D\backslash E)=A^p(D)$ if and only if $C_q(E)=0$,  where $\frac1p+\frac1q=1$.  
\end{theorem}
  
  Here $C_q(E)$ is  the $q-$capacity of $E$ defined by
$
C_q(E):=\inf_\phi \int_{\mathbb C} |\nabla \phi|^q,
$
where the infimum is taken over all $\phi\in C_0^\infty(\mathbb C)$ such that $\phi\ge 1$ on $E$. 

\begin{lemma}\label{lm:Integ}
Let $1<p<\infty$ and $\frac1p+\frac1q=1$.  Let $E$ be a compact set in $\mathbb C$ satisfying $C_q(E)>0$ and $C_{q'}(E)=0$\/ for all $q'<q$.  Then there exists $R_0\gg 1$ such that the following properties hold  for all $R\ge R_0$:
\begin{enumerate}
 \item[$(1)$]  Given $p'>p$, there exists a point $z\in \mathbb D_R\backslash E$ so that $K_{\mathbb D_R\backslash E,p}(\cdot,z)\notin L^{p'}(\mathbb D_R\backslash E)$;
 \item[$(2)$] There exists a point $z\in \mathbb D_R\backslash E$ such that $K_s(z)$ is not continuous at $s=p$.
 \end{enumerate}
\end{lemma}
 
 \begin{proof}
 Since $A^p(\mathbb C)=\{0\}$ for $1<p<\infty$,  it follows from Theorem \ref{th:Hedberg} that $A^p(\mathbb C\backslash E)\neq \{0\}$,  which implies  $K_{\mathbb C\backslash E,p}(z_0)> 0$ for some point $z_0$.  For the sake of simplicity,  we assume $z_0=0$.
 Take $R_0$ so that $E\subset \mathbb D_{R_0}$ and $K_{\mathbb C\backslash E,p}(0)>\frac1{\pi R_0^2}$.  Let $R\ge R_0$ be given.  Set $f_0:=K_{\mathbb D_R\backslash E,p}(\cdot,0)$.  We claim that $f_0\notin L^{p'}(\mathbb D_R\backslash E)$ for any $p'>p$.  Indeed,  suppose on the contrary that $f_0\in L^{p'}(\mathbb D_R\backslash E)$ for some $p'>p$.  Since $C_{q'}(E)=0$ where $\frac1{p'}+\frac1{q'}=1$,  it follows from  Theorem \ref{th:Hedberg} that $f_0$ can be extended holomorphically across $E$,  so that $f_0\in A^p(\mathbb D_R)$.  By the mean-value inequality,  we have
 $$
 K_{\mathbb D_R\backslash E,p}(0)^p=|f_0(0)|^p \le \frac1{\pi R^2} \int_{\mathbb D_R} |f_0|^p= \frac1{\pi R^2} \int_{\mathbb D_R\backslash E} |f_0|^p= \frac1{\pi R^2} K_{\mathbb D_R\backslash E,p}(0)^{p-1},
 $$ 
 so that
 $$
\frac1{\pi R_0^2}\ge \frac1{\pi R^2} \ge K_{\mathbb D_R\backslash E,p}(0)\ge K_{\mathbb C\backslash E,p}(0),
 $$
 which is absurd.
 
 For $s>p$,  we have 
 $$
 K_{\mathbb D_R\backslash E,s}(0)= K_{\mathbb D_R,s}(0) =\frac1{\pi R^2}\le \frac1{\pi R^2_0}<K_{\mathbb D_R\backslash E,p}(0).
 $$
 In particular,  $K_{\mathbb D_R\backslash E,s}(0)$ is not continuous at $s=p$.
 \end{proof}
 
Recall that a compact subset $E\subset\mathbb{C}$ is uniformly perfect (in the sense of Pomerenke \cite{Pommerenke}) if there exists constants $c>0$ and $r_0>0$ such that
$$
E\cap\{z\in\mathbb{C}:\ cr\leq|z-a|\leq{r}\} \neq\emptyset
$$
for all $a\in{E}$ and $0<r<r_0$.  It is known  that for a bounded domain $\Omega\subset\mathbb{C}$,  one has $\alpha(\Omega)>0$ when $\partial\Omega$ is uniformly perfect (cf.  \cite{CT},  Theorem 1.7).  
 
 \begin{proposition}\label{prop:cex}
 For each $1<q<2$,  there exists a uniformly perfect compact set $E$ in $\mathbb C$ such that $C_q(E)>0$ and $C_{q'}(E)=0$ for all $q'<q$. 
 \end{proposition}

\begin{proof}
Let us recall a construction due to Lindqvist \cite{Lindqvist} of a compact set $E$ with $C_q(E)>0$ and $C_{q'}(E)=0$,  $\forall\,q'<q$. Given a sequence $\{l_j\}^\infty_{j=0}$ of positive numbers with $l_{j+1}<l_j/2$, set $\mathcal C_0=[0,l_0]$ and define $\mathcal C_j$ to be a union of $2^j$ closed intevals inductively, such that $\mathcal C_j$ is obtained from $\mathcal C_{j-1}$ with an open inteval of length $l_{j-1}-2l_j$ deleted in the center of each   $2^{j-1}$ closed inteval in $\mathcal C_{j-1}$. For example, $\mathcal C_1=[0,l_1]\cup[l_0-l_1,l_0]$, $\mathcal C_2=[0,l_2]\cup[l_1-l_2,l_1]\cup[l_0-l_1,l_0-l_1+l_2]\cup[l_0-l_2,l_0]$, etc. Write
$$
\mathcal C_j=\bigcup^{2^j}_{k=1}I_{j,k},
$$
where every $I_{j,k}$ is a closed inteval of length $l_j$, lying on the left of $I_{j,k+1}$. The Cantor-type set
$$
\mathcal C:=\bigcap^\infty_{j=0} \mathcal C_j
$$
is what we are looking for.

Now take $l_0=1$, $0<l_1<1/2$, and
$$
l_j=\frac{j^{\frac{q}{2-q}}}{2^{j\frac{2}{2-q}}},\ \ \ j\geq2.
$$
Since
$$
\frac{l_{j+1}}{l_j}=\left(\frac{j+1}{j}\right)^{\frac{q}{2-q}}\left(\frac{1}{2}\right)^{\frac{2}{2-q}},
$$
it follows that
\begin{equation}\label{eq:ratioCantor}
  \left(\frac{1}{2}\right)^{\frac{2}{2-q}}<\frac{l_{j+1}}{l_j}<\frac{1}{2},\ \ \ j\geq2.
\end{equation}
On the other hand, since
$$
\left(\left(\frac{1}{2}\right)^{\frac{2}{2-q}}\right)^2<l_2=2^{-\frac{4-q}{2-q}}<\left(\frac{1}{2}\right)^2,
$$
we may choose $l_1=l_2^{\frac{1}{2}}$ so that \eqref{eq:ratioCantor} remains valid for all $j\geq0$. Thus we obtain a Cantor-type set $\mathcal C$ in $\mathbb{R}$.  By Lemma 7.1 of \cite{Lindqvist},  we know that  the compact set
$$
E:=\mathcal C\times{\mathcal C}
$$
satisfies $C_q(E)>0$ and $C_{q'}(E)=0$ for $q'<q$.

It remains to verify the uniformly perfectness of $E$. By the definition of $E$, this is equivalent to find constants $c>0$ and $r_0>0$, such that
$$
\mathcal C\cap\{x\in\mathbb{R}:\ cr<|x-a|<r\}\neq\emptyset
$$
for all $a\in{\mathcal C}$ and $0<r<r_0$. For $0<r<r_0$, there exists an integer $j$ with $l_{j+1}<r/2<l_j$. By choosing $r_0<2l_2$, we may assume $j\geq2$. Given $j$ and $a\in{\mathcal C}$, there is an integer $k$ depending on $a$ and $j$ such that $a\in{I_{j+1,k}}$. We claim that
\begin{equation}\label{eq:nonemptyintersection}
  I_{j+1,k}\cap\{x\in\mathbb{R}:\ cr<|x-a|<r\}\neq\emptyset.
\end{equation}
Indeed, suppose on the contrary that $I_{j+1,k}\cap\{x\in\mathbb{R}:\ cr<|x-a|<r\}=\emptyset$. Since $a\in{I_{j+1,k}}$, we must have $I_{j+1,k}\subset[a-cr,a+cr]$. Take $c$ so that
$$
0<c<\frac{1}{6}\cdot\left(\frac{1}{2}\right)^{\frac{2}{2-q}}<\frac{1}{2}.
$$
This together with \eqref{eq:ratioCantor} imply that
$$
2cr<4cl_j<4\cdot2^{\frac{2}{2-q}}cl_{j+1}<\frac{2}{3}l_{j+1},
$$
i.e., the length of $[a-cr,a+cr]$ is less than the length of $I_{j+1,k}$, which is a contradiction.

In view of \eqref{eq:nonemptyintersection}, we see that if $\mathcal C\cap\{x\in\mathbb{R}:\ cr<|x-a|<r\}=\emptyset$, then
$$
(I_{j+1,k}\setminus{\mathcal C})\cap\{x\in\mathbb{R}:\ cr<|x-a|<r\}\neq\emptyset.
$$
Thus either $(a-r,a-cr)$ or $(a+cr,a+r)$ must be contained in a connected component of $I_{j+1,k}\setminus{\mathcal C}$. By the definition of $\mathcal C$, such a connected component is an open inteval of length $l_m-2l_{m+1}$ for $m\geq{j+1}$. In particular, the length is no more than $l_m\leq{l_{j+1}}$. On the other hand, the length of $(a-r,a-cr)$ or $(a+cr,a+r)$ equals to $(1-c)r$, which is no less than $r/2>l_{j+1}$ since $0<c<1/2$. This leads to a contradiction.
\end{proof}

\begin{proof}[Proof of Proposition \ref{prop:c-example}]
The conclusion follows directly from Proposition \ref{prop:cex} and Lemma \ref{lm:Integ}.
\end{proof}

\begin{proof}[Proof of Proposition \ref{prop:p>2}]
By Proposition \ref{prop:cex},  we have a uniformly perfect compact set $E\subset \mathbb C$ such that $C_q(E)=0$ where $\frac1p+\frac1q=1$. Take $R>0$ such that $E\subset \mathbb D_R$.  Set $\Omega:=\mathbb D_\mathbb R\backslash E$.  For the sake of simplicity,  we assume $R=1$.  By Theorem \ref{th:Hedberg},  we conclude that for $\zeta,z\in \Omega$,
\begin{eqnarray*}
m_{\Omega,p}(\zeta,z) & = & m_{\mathbb D,p}(\zeta,z)= \left(\frac{1-|z|^2}{1-\bar{z}\zeta}\right)^{4/p}\\
K_{\Omega,p}(\zeta,z) & = & K_{\mathbb D,p}(\zeta,z)=\frac1\pi \left(\frac{1-|z|^2}{1-\bar{z}\zeta}\right)^{4/p}(1-|z|^2)^{-2}
\end{eqnarray*}
in view of \cite{CZ},  Proposition 2.9.  Fix a domain $U$ such that $E\subset U\subset\subset \mathbb D$.  Then there exits a constant $C>1$ such that for $z\in U\backslash E$ and $\zeta\in \Omega$,
$$
 C^{-1}\le |m_{\Omega,p}(\zeta,z)| \le C,\ \ \  C^{-1}\le |K_{\Omega,p}(\zeta,z)| \le C,
$$
and consequently,  
$$
C^{-1} K_{\Omega,2}(z) \le K_{\Omega,2,p,z} (z) \le C K_{\Omega,2}(z).
$$
Since $\Omega$ is hyperconvex,  so $K_{\Omega,2}(z)$ is an exhaustion function on $\Omega$ in view of Ohsawa \cite{Ohsawa}.  Thus $\lim_{z\rightarrow E} K_{\Omega,2,p,z} (z)=+\infty$ while $\limsup_{z\rightarrow E} K_{\Omega,p} (z)<+\infty$,  from which the assertion immediately follows. 
\end{proof}

Another consequence of Theorem \ref{th:Hedberg}  and Proposition \ref{prop:cex} is the following

\begin{proposition}
For each $2<p<\infty$,  there exists a bounded domain $\Omega\subset \mathbb C$ with $\alpha(\Omega)>0$ such that $K_p(z)$ is not an exhaustion function on $\Omega$.  
\end{proposition}

\section{Curvature properties of  $B_p(z;X)$}
Let $B_p(z;X)$ be the $p-$Bergman metric on a bounded domain $\Omega\subset \mathbb C^n$.  The\/ {\it holomorphic sectional curvature} of the\/ {\it Finsler} metric $B_p(z;X)$ is given by
$$
\mathrm{HSC}_p(z;X):=\sup_{\sigma}\left\{ \frac{\Box \log B_p^2(\sigma;\sigma')(0)}{-B_p^2(z;X)^2} \right\},
$$ 
where the supremum is taken over all holomorphic mappings 
$
\sigma: \mathbb D_r\rightarrow \Omega
$
 such that $\sigma(0)=z$ and $\sigma'(0)=X$.  Recall from \cite{CZ} that $B_p(z;X)=K_p(z)^{-1/p} \mathcal M_p(z;X)$,  where 
$$
\mathcal M_p(z;X)=\sup\left\{|Xf(z)|: f\in A^p(\Omega),\,f(z)=0,\, \|f\|_p=1  \right\},
$$
$Xf := \sum_j X_j \partial f/\partial z_j$.  
Note  that $\mathcal M_p(z;X)=1/m_p(z;X)$,  where 
$$
m_p(z;X)=\inf\left\{\|f\|_p: f\in A^p(\Omega),\,f(z)=0,\,Xf(z)=1 \right\}.  
$$ 

\begin{proof}[Proof of Theorem \ref{th:HSC_0}]
Note that
\begin{eqnarray}\label{eq:HSC_2}
\mathrm{HSC}_p(z;X) & \le &  \sup_\sigma\left\{\frac2p\cdot \frac{\Box \log K_p\circ \sigma(0)}{B_p(z;X)^2} \right\} + \sup_\sigma\left\{ - \frac{\Box \log\mathcal M_p^2(\sigma;\sigma')(0)}{B_p(z;X)^2} \right\}\nonumber\\
& = & \frac2p\cdot \frac{ i\partial\bar{\partial} \log K_p(z;X) }{B_p(z;X)^2} + \sup_\sigma\left\{ - \frac{\Box \log\mathcal M_p^2(\sigma;\sigma')(0)}{B_p(z;X)^2} \right\}
\end{eqnarray}
in view of Lemma \ref{lm:Invariant_1}.  Fix  $\sigma$ for a moment.
Recall that 
\begin{eqnarray*}
&& \Box \log\mathcal M_p^2(\sigma;\sigma')(0)\\
& = & \liminf_{r\rightarrow 0+}\frac1{r^2}\left\{\frac1{2\pi}\int_0^{2\pi} \log \mathcal M_p^2(\sigma;\sigma')(re^{i\theta})d\theta-\log \mathcal M_p^2(z;X)\right\}.
\end{eqnarray*}
Take $f_0,f_1\in A^p(\Omega)$ such that
\begin{enumerate}
\item $f_0(z)=1$ and $\|f_0\|_p=m_p(z)$;
\item $f_1(z)=0,\,Xf_1(z)=1$ and $\|f_1\|_p=m_p(z;X)$;
\end{enumerate}
Fix $z,r,\theta$ for a moment.  Consider the following family of holomorphic functions:
$$
\widetilde{f}_{t}=f_1+t f_0,\ \ \ t\in \mathbb C.
$$ 
Now fix
$$
t:=-\frac{f_1 ( \sigma(re^{i\theta}))}{f_0 (\sigma(re^{i\theta})) }=-re^{i\theta}+O(r^2).
$$
Then we have 
$
\widetilde f_{t} ( \sigma(re^{i\theta})) =0.
$
For  $g\in C^1(\Omega)$ and $\sigma=(\sigma_1,\cdots,\sigma_n)$,  we define
$$
\partial_\sigma g(re^{i\theta}):= \sum_{j=1}^n \sigma_j'(re^{i\theta})\frac{\partial g}{\partial \zeta_j}(\sigma(re^{i\theta})).   
$$
Then
 we have 
\begin{eqnarray*}
|\partial_\sigma \widetilde f_t (re^{i\theta}) |^2 & = & |\partial_\sigma f_1(re^{i\theta})|^2  +|t|^2 |\partial_\sigma f_0(re^{i\theta}) |^2\\
&& + 2\mathrm{Re}\left\{\partial_\sigma f_1(re^{i\theta})\,\overline{t \partial_\sigma f_0(re^{i\theta})} \right\};\\
| \partial_\sigma f_1(re^{i\theta}) |^2 &= & |1+a_1 re^{i\theta} +c(r e^{i\theta})^2+O(r^3)|^2\\
& = & 1 +  |a_1|^2 r^2   + 2\mathrm{Re}\left\{ a_1 re^{i\theta} +c(r e^{i\theta})^2 \right\} +O(r^3); \\
  |t|^2 | \partial_\sigma f_0(re^{i\theta}) |^2 & = &  |Xf_0(z)|^2 r^2 +O(r^3)
\end{eqnarray*}
and 
\begin{eqnarray*}
&& 2\mathrm{Re}\left\{\partial_\sigma f_1(re^{i\theta})\,\overline{t \partial_\sigma f_0(re^{i\theta})} \right\}\\
& = & 2\mathrm{Re}\left\{\left(1+  a_1 re^{i\theta}\right) \overline{t\left(Xf_0(z) + a_0 re^{i\theta} \right)} \right\} +O(r^3)\\
& = & -2\mathrm{Re}\left\{ a_1 \overline{ Xf_0(z) } \right\}r^2   + 2\mathrm{Re}\left\{ t \left( Xf_0(z)+a_0 re^{i\theta}\right)\right\}+O(r^3).
\end{eqnarray*}
Here and in what follows $a_j,b_j,c_j,c$ are complex numbers. 
Thus
\begin{eqnarray}\label{eq:HSC_3}
|\partial_\sigma \widetilde f_t (re^{i\theta}) |^2  & \ge & 1 +\Phi(r,\theta)  + |Xf_0(z)-a_1|^2 r^2+O(r^3),
\end{eqnarray}
where 
$$
\Phi(r,\theta)=2\mathrm{Re}\left\{ a_1 re^{i\theta} +c(r e^{i\theta})^2+ t \left( Xf_0(z)+ a_0 re^{i\theta}\right)\right\}.
$$
Since 
$$
t=-\frac{re^{i\theta}+c_1 (re^{i\theta})^2+O(r^3)}{1+b_1 re^{i\theta}+ O(r^2)}=-re^{i\theta}+(b_1-c_1)(re^{i\theta})^2+O(r^3),
$$
it follows that
\begin{eqnarray*}
\Phi(r,\theta) & = & 2\mathrm{Re}\left\{ \left(a_1 -Xf_0(z)\right)re^{i\theta} \right\} + 2\mathrm{Re}\left\{ \left(c+(b_1-c_1)Xf_0(z)- a_0\right)(re^{i\theta})^2\right\}+O(r^3).
\end{eqnarray*}
Thus
\begin{eqnarray}
 \frac1{2\pi} \int_0^{2\pi} \Phi(r,\theta) d\theta & = & O(r^3), \label{eq:HCS_3_1}\\  
 \frac1{2\pi} \int_0^{2\pi} \Phi(r,\theta)^2 d\theta & = & 2 |Xf_0(z)- a_1|^2 r^2+O(r^3).\label{eq:HCS_3_2}
\end{eqnarray}
Use the expansion $\log(1+x)=x-x^2/2+O(|x|^3)$,  we infer from $\eqref{eq:HSC_3}\sim \eqref{eq:HCS_3_2}$ that
\begin{eqnarray}
&&\frac1{2\pi} \int_0^{2\pi} \log |\partial_\sigma \widetilde f_t (re^{i\theta}) |^2 d \theta\nonumber\\
 & \ge & |Xf_0(z)-a_1|^2 r^2 +\frac1{2\pi} \int_0^{2\pi} \Phi(r,\theta)d\theta\nonumber\\
 &&  -\frac1{4\pi} \int_0^{2\pi}
\Phi(r,\theta)^2d\theta+O(r^3).\nonumber\\
&=& O(r^3)\label{eq:HSC_3-3}
\end{eqnarray}

On the other hand,  we have
\begin{eqnarray*}
J(t):=\|\widetilde f_{t}\|_p^p  =  J(0)+2\mathrm{Re}\left\{ \frac{\partial J}{\partial t}(0) t + \frac12 \frac{\partial^2 J}{\partial t^2}(0) t^2\right\} + \frac{\partial^2 J}{\partial t\partial \bar{t}}(0) |t|^2+o(r^2)
\end{eqnarray*}
where
\begin{eqnarray*}
\frac{\partial J}{\partial t}(0)  & = &  \frac{p}2 \int_\Omega |f_1|^{p-2} \overline{f_1}f_0,\\
\frac{\partial^2 J}{\partial t^2}(0) & = & \frac{p(p-2)}4 \int_\Omega |f_1|^{p-4} \overline{f_1^2} f_0^2,\\
\frac{\partial^2 J}{\partial t\partial \bar{t}} (0) & = & \frac{p^2}4 \int_\Omega |f_1|^{p-2} |f_0|^2.
\end{eqnarray*}  
  Since 
\begin{eqnarray*}
 \int_\Omega |f_1|^{p-2} |f_0|^2 & \le & \left(\int_\Omega |f_1|^p\right)^{\frac{p-2}p}\left( \int_\Omega |f_0|^p\right)^{\frac2p}=m_p(z;X)^{p-2} m_p(z)^2
\end{eqnarray*}
in view of H\"older's inequality,  we obtain
\begin{eqnarray}\label{eq:HSC_4}
\|\widetilde f_{t}\|_p^p & \le & m_p(z;X)^p +\Psi(r,\theta) +  \frac{p^2}4 m_p(z;X)^{p-2} m_p(z)^2 r^2 +o(r^2),
\end{eqnarray}
where 
\begin{eqnarray*}
\Psi(r,\theta) & = & 2\mathrm{Re}\left\{ \frac{\partial J}{\partial t}(0) t + \frac12 \frac{\partial^2 J}{\partial t^2}(0) t^2 \right\}
\end{eqnarray*}
satisfies
\begin{eqnarray}
\frac1{2\pi}\int_0^{2\pi} \Psi(r,\theta)d\theta & = &  0\label{eq:HSC_4_1}\\ 
\frac1{2\pi} \int_0^{2\pi} \Psi(r,\theta)^2 d\theta & = & 2 \left| \frac{\partial J}{\partial t}(0) \right|^2 r^2 + O(r^3)\ge O(r^3).\label{eq:HSC_4_2}
\end{eqnarray}
Use again the expansion $\log(1+x)=x-x^2/2+O(|x|^3)$,  we infer from $\eqref{eq:HSC_4}\sim \eqref{eq:HSC_4_2}$ that
\begin{eqnarray}\label{eq:HSC_5}
\frac1{2\pi} \int_0^{2\pi} \log \| \widetilde f_{t}\|_p^2\, d\theta &\le & \log m_p(z;X)^2 + \frac{pr^2}2 \frac{m_p(z)^2}{m_p(z;X)^2} +\frac2p\cdot \frac1{2\pi}\int_0^{2\pi}\frac{\Psi(r,\theta)}{m_p(z;X)^p}d\theta\nonumber\\
&& -\frac2p\cdot  \frac1{4\pi}\int_0^{2\pi} \frac{\Psi(r,\theta)^2}{m_p(z;X)^{2p}}d\theta+ o(r^2)\nonumber\\
& \le &  \log m_p(z;X)^2 + \frac{pr^2}2 \frac{m_p(z)^2}{m_p(z;X)^2} + o(r^2).
\end{eqnarray}
Use $\widetilde f_{t}$ as the test mapping,  we infer from \eqref{eq:HSC_3-3} and \eqref{eq:HSC_5} that
\begin{eqnarray*}
&& \frac1{2\pi} \int_0^{2\pi} \log \mathcal M_p^2(\sigma;\sigma')(re^{i\theta}) d\theta \\
& \ge & \log \frac1{m_p(z;X)^2} -\frac{p}2 \, \frac{ m_p(z)^2}{ m_p(z;X)^2}\,  r^2 + o(r^2).
\end{eqnarray*}
This together with \eqref{eq:HSC_2} yield \eqref{eq:HSC_0}.
\end{proof}

\section{Bergman meets Hardy}
Throughout this section,  $\Omega$ is always a bounded domain with $C^2-$boundary.  
Let $\rho$ be a $C^2$ defining function on $\Omega$.  Following Stein \cite{Stein},  we define the Hardy space $H^p(\Omega)$  to be the set of $f\in \mathcal O(\Omega)$ satisfying
  $$
  \|f\|_{H^p(\Omega)}^p:=\sup_{\varepsilon>0} \int_{\rho=-\varepsilon} |f|^p d S_\varepsilon <\infty.
  $$
  It is known that the set $H^p(\Omega)$ is independent of the choice of defining functions.  In what follows,  we fix $\rho=-\delta$,  where $\delta$ is the boundary distance.  

Given $\varepsilon>0$ and $f\in \mathcal O(\Omega)$,  define $\Omega_\varepsilon:=\{z\in \Omega:\delta(z)>\varepsilon\}$ and $M_p(\varepsilon,f):=\|f\|_{L^p(\partial \Omega_\varepsilon)}$.  
Let $\pi_\varepsilon$ be the normal projection from $\partial \Omega_\varepsilon$ to $\partial \Omega$.  Without loss of generality,  we assume that $\pi_\varepsilon$ is well-defined for $0< \varepsilon \le 1$.  
It is easy to see that
$$
M_p(\varepsilon,f)=\left(\int_{\partial\Omega} |f\circ \pi_\varepsilon^{-1}|^p |\textrm{det}(\pi_\varepsilon^{-1})'|  dS\right)^{1/p} \asymp \left(\int_{\partial\Omega} |f\circ \pi_\varepsilon^{-1}|^p dS\right)^{1/p}=: \widetilde{M}_p(\varepsilon,f).
$$

\begin{proof}[Proof of Theorem \ref{th:HL_2}]
Step 1.  Let  $-1<b<\infty$.  Integration by parts gives
\begin{eqnarray}\label{eq:HL_2}
\int_0^1 \varepsilon^b \widetilde{M}_q(\varepsilon,f)^q d\varepsilon & = & \frac1{b+1}\int_0^1 \widetilde{M}_q(\varepsilon,f)^q d\varepsilon^{b+1}\\
& = &  \frac1{b+1} \widetilde{M}_q(1,f)^q-\frac1{b+1}\int_0^1 \varepsilon^{b+1} \frac{\partial}{\partial \varepsilon} \widetilde{M}_q(\varepsilon,f)^q d\varepsilon.\nonumber
\end{eqnarray}
Since
\begin{eqnarray*}
\left| \frac{\partial}{\partial \varepsilon} \widetilde{M}_q(\varepsilon,f)^q \right|  & = &  \left| \frac{\partial}{\partial \varepsilon} \int_{\partial\Omega} |f\circ \pi_\varepsilon^{-1}|^q  dS  \right|\\
& \le & \int_{\partial \Omega} \left| \frac{\partial}{\partial \varepsilon} |f\circ \pi_\varepsilon^{-1}|^q \right|dS \\
& \lesssim & \int_{\partial \Omega} |f\circ \pi_\varepsilon^{-1}|^{q-1} |\nabla f\circ \pi_\varepsilon^{-1}| dS \\
& \lesssim & \int_{\partial \Omega_\varepsilon} |f|^{q-1} |\nabla f| dS_\varepsilon,
\end{eqnarray*}
it follows that
\begin{eqnarray*}
 \left|\int_0^1 \varepsilon^{b+1} \frac{\partial}{\partial \varepsilon} \widetilde{M}_q(\varepsilon,f)^q d\varepsilon \right|
& \lesssim & \int_0^1 \varepsilon^{b+1}  \int_{\partial \Omega_\varepsilon} |f|^{q-1} |\nabla f| dS_\varepsilon d\varepsilon\\
& \lesssim & \text{small const.} \int_0^1  \varepsilon^b M_q(\varepsilon,f)^q d\varepsilon \\
&& +\, \text{large const.} \int_0^1  \varepsilon^{b+2} \int_{\partial \Omega_\varepsilon} |f|^{q-2} |\nabla f|^2 dS_\varepsilon d\varepsilon,
\end{eqnarray*}
in view of the Cauchy-Schwarz inequality.   Since $\widetilde{M}_q(\varepsilon,f)\asymp M_q(\varepsilon,f)$,  it follows from \eqref{eq:HL_2} that
\begin{eqnarray}\label{eq:HL_3}
\int_0^1  \varepsilon^b M_q(\varepsilon,f)^q d\varepsilon \lesssim  {M}_q(1,f)^q+ \int_0^1  \varepsilon^{b+2} \int_{\partial \Omega_\varepsilon} |f|^{q-2} |\nabla f|^2 dS_\varepsilon d\varepsilon.
\end{eqnarray}

Step 2.  Let $q>p$.  Apply the mean-value inequality on suitable polydisc with center $z$,  we obtain  
\begin{equation}\label{eq:HL_5}
|f(z)|^p \lesssim \varepsilon^{-n-1} \int_{\varepsilon/2\le \delta \le 2\varepsilon} |f|^p\lesssim \varepsilon^{-n} \|f\|^p_{H^p(\Omega)}, 
\end{equation}
that is
\begin{equation}\label{eq:HL_6}
|f(z)| \lesssim \varepsilon^{-n/p} \|f\|_{H^p(\Omega)}.
\end{equation}

Step 3.  Let $0<p<q$.  By \eqref{eq:HL_6},  we have
\begin{eqnarray*}
 \int_{\partial \Omega_\varepsilon} |f|^{q-2} |\nabla f|^2 dS_\varepsilon & = &  \int_{\partial \Omega_\varepsilon} |f|^{q-p} |f|^{p-2} |\nabla f|^2 dS_\varepsilon \\
 & \lesssim & \varepsilon^{-n(\frac{q}p-1)}\|f\|_{H^p(\Omega)}^{q-p} \cdot \int_{\partial \Omega_\varepsilon}  |f|^{p-2} |\nabla f|^2 dS_\varepsilon.
\end{eqnarray*}
This together with \eqref{eq:HL_3} yield
$$
\int_0^1  \varepsilon^b M_q(\varepsilon,f)^q d\varepsilon \lesssim  {M}_q(1,f)^q+\|f\|_{H^p(\Omega)}^{q-p} \int_0^1  \varepsilon^{b+2-n(\frac{q}p-1)} \int_{\partial \Omega_\varepsilon} |f|^{p-2} |\nabla f|^2 dS_\varepsilon d\varepsilon.
$$
Take $b=0$ and $q=p(1+\frac1n)$,  we obtain 
\begin{eqnarray}\label{eq:HL_8}
\int_\Omega |f|^q & = & \int_{\Omega_1} |f|^q + \int_0^1   M_q(\varepsilon,f)^q d\varepsilon\nonumber\\
& \lesssim & \int_{\Omega_1} |f|^q + {M}_q(1,f)^q+\|f\|_{H^p(\Omega)}^{q-p} \int_0^1  \varepsilon\int_{\partial \Omega_\varepsilon} |f|^{p-2} |\nabla f|^2 dS_\varepsilon d\varepsilon\nonumber\\
& \lesssim & \int_{\Omega_1} |f|^q + {M}_q(1,f)^q+\|f\|_{H^p(\Omega)}^{q-p}\int_\Omega \delta |f|^{p-2} |\nabla f|^2\nonumber\\
& \lesssim & \int_{\Omega_1} |f|^q + {M}_q(1,f)^q+\|f\|_{H^p(\Omega)}^{q},
\end{eqnarray}
in view of  Proposition \ref{prop:Carleson} given below.   

To get \eqref{eq:HL_7},  it suffices to show that the identity mapping $I:H^p(\Omega)\rightarrow A^q(\Omega)$ is a continuous linear functional.   To see this,  take any $\{f_\nu\}\subset H^p(\Omega)$  with $f_\nu \rightarrow f_0$ in $H^p(\Omega)$.  By the Bergman inequality,  we see that $f_\nu$ converges uniformly on $\overline{\Omega}_1$ to $f_0$.   It follows immediately from \eqref{eq:HL_8} that $\|f_\nu-f_0\|_{A^q(\Omega)}\rightarrow 0$ as $\nu\rightarrow \infty$.
\end{proof}

\begin{proposition}\label{prop:Carleson}
Let $\Omega$ be a bounded domain with $C^2-$boundary in $\mathbb C^n$.  Then
\begin{equation}\label{eq:HL_4}
\int_\Omega \delta |f|^{p-2} |\nabla f|^2 \lesssim \|f\|^p_{H^p(\Omega)},\ \ \ \forall\,f\in H^p(\Omega).
\end{equation}
\end{proposition}

\begin{proof}
The argument is standard.  First of all,  it is fairly easy to construct a subharmonic defining function $\rho$ on $\Omega$ (see e.g.,  \cite{ChenXing},  Lemma 2.1).  Let $\varphi$ be a $C^2$ subharmonic function on $\Omega$.  Green's theorem gives
\begin{eqnarray*}
\int_{\rho<-\varepsilon}(-\rho-\varepsilon)(\Delta \varphi+|\nabla \varphi|^2) e^\varphi
& = & \int_{\rho<-\varepsilon}(-\rho-\varepsilon) \Delta e^\varphi\\
 & = & -\int_{\rho<-\varepsilon} \Delta\rho\,e^\varphi +\int_{\rho=-\varepsilon} \frac{\partial \rho}{\partial \nu_\varepsilon} e^\varphi dS_\varepsilon\\
 & \lesssim & \int_{\rho=-\varepsilon} e^\varphi.
\end{eqnarray*}
Take first $\varphi=\frac{p}2\log (|f|^2+\tau)$ then let $\tau\rightarrow 0+$ and $\varepsilon\rightarrow 0+$,  we immediately obtain \eqref{eq:HL_4}.
\end{proof}

\section{The $p-$Schwarz content and its applications}
\subsection{Basic properties and upper bounds}\label{subsec:basic_property}
We first list a few trivial properties of $s_p$ for $0<p<\infty$ as follows:
\begin{enumerate}
\item[$(a)$] $s_p(E,\Omega)\le 1$ and $s_p(\Omega,\Omega)=1$;
\item[$(b)$] $E_1\subset E_2$ implies $s_p(E_1,\Omega)\le s_p(E_2,\Omega)$,  and $\Omega_1\subset \Omega_2$ implies $s_p(E,\Omega_1)\ge s_p(E,\Omega_2)$;
\item[$(c)$] Subadditivity:  $s_p(\bigcup_{j=1}^\infty E_j,\Omega)\le \sum_{j=1}^\infty s_p(E_j,\Omega)$;
\item[$(e)$]  $s_p(E,\Omega)\ge |E|/|\Omega|$,  where $|\cdot|$ stands for the volume;
\item[$(f)$]  $s_p(E,\Omega)\le C_n  |E|/d^{2n}$ where $d=d(E,\partial \Omega)$; this follows from the Bergman inequality: $|f(z)|^p\le C_n \delta(z)^{-2n}\int_\Omega |f|^p$ for $f\in A^p(\Omega)$.  
\end{enumerate}

\begin{example}
 $s_p(\mathbb D_r,\mathbb D)=r^{2}$.  
\end{example}

To see this,  first note that $s_p(\mathbb D_r,\mathbb D)\ge |\mathbb D_r|/|\mathbb D|=r^{2}$; on the other side,  since  $M_p(r,f)^p:=\int_0^{2\pi} |f(re^{i\theta})|^p d\theta$ is non-decreasing in $r$,  it follows that 
$$
\int_{\mathbb D_r}|f|^p = \int_0^r t M_p(t,f)^p dt\le \frac{r^2}2 M_p(r,f)^p \le \frac{r^2}{1-r^2} \int_r^1 t M_p(t,f)^p dt=\frac{r^2}{1-r^2}\int_{\mathbb D\backslash \mathbb D_r} |f|^p,
$$
i.e.,  $\int_{\mathbb D_r}|f|^p\le r^2 \int_{\mathbb D}|f|^p$,  so that $s_p(\mathbb D_r,\mathbb D)\le r^{2}$.

\begin{proposition}
If\/ $\Omega$ is a simply-connected domain in $\mathbb C^n$ and $F:\Omega\rightarrow \mathbb C^n$ is a holomorphic injective mapping,  then 
$$
s_p(E,\Omega)=s_p(F(E),F(\Omega)).
$$
Moreover, the simply-connected condition can be removed when $p=2/m$, where $m\in\mathbb{Z}^+$.
\end{proposition}

\begin{proof}
Set $E'=F(E)$ and $\Omega'=F(\Omega)$. Then we have
$$
f'\in A^p(\Omega') \iff f:=f'\circ F\cdot J_F^{2/p}\in A^p(\Omega),
$$
so that 
$$
\frac{\int_{E'}|f'|^p}{\int_{\Omega'}|f'|^p} = \frac{\int_{E}|f|^p}{\int_{\Omega}|f|^p}\le s_p(E,\Omega).
$$
Take supremum over $f'\in A^p(\Omega')$,  we have $s_p(E',\Omega')\le s_p(E,\Omega)$.  Consider $F^{-1}$ instead of $F$,  we obtain the reverse inequality.
\end{proof}

\begin{proposition}\label{prop:mono}
Let $\Omega'\subset \Omega$ be two bounded domains in $\mathbb C^n$.  Then we have 
$$
\frac{K_{\Omega,p}(z)}{K_{\Omega',p}(z)}\le s_p(\Omega',\Omega),\ \ \ \forall\,z\in \Omega'.
$$
\end{proposition}

\begin{proof}
For every $z\in \Omega'$,  we have
\begin{eqnarray*}
K_{\Omega,p}(z) & = & \sup_{f\in A^p(\Omega)\backslash \{0\}}\left\{ \frac{|f(z)|^p}{\int_{\Omega}|f|^p}\right\}\\
& = & \sup_{f\in A^p(\Omega)\backslash \{0\}} \left\{\frac{|f(z)|^p}{\int_{\Omega'}|f|^p} \cdot \frac{\int_{\Omega'}|f|^p}{\int_{\Omega}|f|^p}\right\}\\
& \le & s_p(\Omega',\Omega) \sup_{f\in A^p(\Omega')\backslash \{0\}}\left\{\frac{|f(z)|^p}{\int_{\Omega'}|f|^p}\right\}\\
& = & s_p(\Omega',\Omega) K_{\Omega',p}(z).
\end{eqnarray*}
\end{proof}

\begin{corollary}
Let $\Omega'\subset \Omega$ be  bounded domains in $\mathbb C^n$ with $n\ge 2$,  such that $\Omega\setminus \Omega'$ is compact in $\Omega$.  Then $s_p(\Omega',\Omega)=1$ for all $0<p<\infty$.  
\end{corollary}

\begin{proof}
Since $\Omega$ is bounded,  there exists at least one peak point $z_0$ on $\partial \Omega$,  i.e.,  there is $h\in \mathcal O(\Omega)$ such that $|h|<1$,   $|h(z)|\rightarrow 1$ as $z\rightarrow z_0$ and $\sup_{\Omega\setminus B(z_0,r)}|h|<1$ for all $r>0$; for instance,  one may take $z_0$ to be a  point on $\partial \Omega$ which is the farthest to the origin.   Since $s_p(\Omega',\Omega)\le 1$,  it suffices to verify that 
\begin{equation}\label{eq:s_p=1}
\lim_{z\rightarrow z_0} \frac{K_{\Omega,p}(z)}{K_{\Omega',p}(z)}=1
\end{equation}
in view of Proposition \ref{prop:mono}.  
To see this,  we take for any given $z\in \Omega'$ a function $f_z\in A^p(\Omega')$ with $|f_z(z)|^p=K_{\Omega',p}(z)$ and $\int_{\Omega'}|f_z|^p=1$.  By the Hartogs extension theorem,  we see that $f_z\in A^p(\Omega)$.  Moreover,  it is easy to see from the maximun principle and the Bergman inequality that 
$$
\int_{\Omega\setminus \Omega'} |f_z|^p \le C \int_{\Omega'} |f_z|^p=C
$$
where $C=C(\Omega'.\Omega)>0$.  
Given $m\in \mathbb Z^+$ and $z\in \Omega$,  define $g_{m,z}:=h^m f_z$.  Clearly,  $g_{m,z}$ is holomorphic on $\Omega$ and satisfies
$$
\int_\Omega |g_{m,z}|^p= \int_{\Omega\setminus \Omega'} |g_{m,z}|^p + \int_{\Omega'} |g_{m,z}|^p \le C \left( \sup_{\Omega\setminus \Omega'}|h|\right)^m + 1<1+\varepsilon,
$$ 
provided $m\ge m_\varepsilon\gg 1$.  Thus
$$
\varliminf_{z\rightarrow z_0} \frac{K_{\Omega,p}(z)}{K_{\Omega',p}(z)}\ge \varliminf_{z\rightarrow z_0} \frac1{K_{\Omega',p}(z)} \frac{|g_{m_\varepsilon,z}(z)|^p}{\int_\Omega |g_{m_\varepsilon,p}|^p}
\ge  \varliminf_{z\rightarrow z_0} \frac{|h(z)|^{m_\varepsilon}}{1+\varepsilon} =\frac1{1+\varepsilon}.
$$
This together with the trivial inequality $K_{\Omega,p}(z)\le K_{\Omega',p}(z)$ yield \eqref{eq:s_p=1}.   

\end{proof}

Below we give an example of a pair $\Omega'\subset \Omega$ of bounded domains in $\mathbb C^n$ such that $\partial \Omega'$ intersects $\partial \Omega$ at two points while $s_p(\Omega',\Omega)=1$.   Define 
\begin{eqnarray*}
U & := & \left\{(z',z_n)\in \mathbb C^{n-1}\times \mathbb C: \mathrm{Im}\,z_n>|z'|^2\right\}\\
U' & := & \left\{(z',z_n)\in \mathbb C^{n-1}\times \mathbb C: \mathrm{Im}\,z_n>|z'|^2+ |z'|^4\right\}.   
\end{eqnarray*}
The Cayley transformation maps $U$ and $U'$ biholomorphically to the unit ball $\mathbb B^n$ and a domain $\Omega'\subset \mathbb B^n$ with $\partial \Omega'\cap \partial \mathbb B^n$ consisting of two points respectively.  In particular,  both $K_{U,p}(z)$ and $K_{U',p}(z)$ are positive for each $z\in U'$,  in view of the transformation formula for $K_p$ (cf.  \cite{CZ},  Proposition 2.7,  which naturally generalizes to unbounded cases).   Clearly,  $F_\varepsilon: (z',z_n)\mapsto (z'/\sqrt{\varepsilon},z_n/\varepsilon)$,  $\varepsilon>0$,  is an automorphism of $U$ and we have
\begin{eqnarray*}
K_{U,p}((0',\varepsilon i)) & = &  K_{U,p}((0', i)) |J_{F_\varepsilon}(0',\varepsilon i)|^2\\
K_{U',p}((0',\varepsilon i)) & = &  K_{F_\varepsilon(U'),p}((0', i)) |J_{F_\varepsilon}(0',\varepsilon i)|^2.
\end{eqnarray*}
On the other hand,  since $F_\varepsilon(U')=\{(z',z_n): \mathrm{Im}\,z_n> |z'|^2+ \varepsilon |z'|^4\}\uparrow U$ as $\varepsilon\downarrow 0$,  a standard normal family argument yields 
$$
\lim_{\varepsilon\rightarrow 0+} K_{F_\varepsilon(U'),p}((0', i))= K_{U,p}((0', i)),
$$
so that 
$$
\lim_{\varepsilon\rightarrow 0+} \frac{K_{U,p}((0',\varepsilon i))}{K_{U',p}((0',\varepsilon i))}= 1.
$$
Thus we have $\sup_{z\in \Omega'}\frac{K_{\mathbb B^n,p}(z)}{K_{\Omega',p}(z)}\ge 1$,  so that $s_p(\Omega',\mathbb B^n)\ge 1$ in view of Proposition \ref{prop:mono}.  Since $s_p(\Omega',\mathbb B^n)\le 1$,  we obtain $s_p(\Omega',\mathbb B^n) = 1$.

\begin{proof}[Proof of Proposition \ref{prop:BM}]
By Banach's open mapping theorem,  the identity mapping
$$
I: A^p(\Omega)\rightarrow A^p(\Omega'),\ \ \ f\mapsto f|_{\Omega'},
$$
is a continuous isomorphism,  which satisfies $\|I\|\le 1$ and
$$
\|I^{-1}\|=\sup \left\{\frac{\|f\|_{L^p(\Omega)}}{\|f\|_{L^p(\Omega')}}:f\in A^p(\Omega)\backslash \{0\}\right\}.
$$
It follows from the definition of the Banach-Mazur distance that
$$
d_{\mathrm{BM}}(A^p(\Omega'),A^p(\Omega))\le \sup \left\{\frac{\|f\|_{L^p(\Omega)}}{\|f\|_{L^p(\Omega')}}:f\in A^p(\Omega)\backslash \{0\}\right\}.
$$
Since 
$$
\int_{\Omega} |f|^p=\int_{\Omega'}|f|^p + \int_{\Omega\backslash \Omega'}|f|^p\le 
\int_{\Omega'}|f|^p+ s_p(\Omega\backslash \Omega',\Omega) \int_{\Omega} |f|^p,
 $$
 we immediately get \eqref{eq:BM}.
\end{proof}

\begin{proposition}\label{prop:p-Schwarz}
Let $\Omega$ be a bounded  domain in $\mathbb C^n$  and $E$ a measurable relatively compact subset in $\Omega$.  Then for any $0<p<\infty$,
\begin{equation}\label{eq:p-SchwarzUpper}
s_p(E,\Omega)\le \frac{136/\lambda_1(\Omega)}{136/\lambda_1(\Omega)+d^2}
\end{equation}
where $d:=d(E,\partial \Omega)$ and $\lambda_1(\Omega)$ denotes the first eigenvalue of the Laplacian on $\Omega$, i.e. ,
$$
\lambda_1(\Omega):=\inf\left\{\frac{\int_\Omega |\nabla \phi|^2}{\int_\Omega \phi^2}: \phi\in W^{1,2}_0(\Omega)\backslash \{0\}\right\}.
$$
\end{proposition}

It follows from
Sobolev's inequality that if $n>1$ then $\lambda_1(\Omega)\ge C_n^{-1} |\Omega|^{-1/n}$ holds for some constant $C_n>0$ depending only on $n$.  Thus

\begin{corollary}
For $n>1$ there exists a constant $C_n>0$ such that
\begin{equation}\label{eq:p-SchwarzUpper2}
s_p(E,\Omega)\le \frac{ C_n |\Omega|^{\frac1n} }{ C_n |\Omega|^{\frac1n}+ d^{2}}.
\end{equation}
\end{corollary}

Let us  recall the following well-known inequality:

\begin{proposition}[Caccioppoli inequality]\label{prop:Caccioppoli}
Let $\Omega$ be a domain in $\mathbb C^n$ and $\psi$ a nonnegative subharmonic function on $\Omega$.  Then
\begin{equation}\label{eq:Caccioppoli}
\int_\Omega \phi^2 |\nabla \psi|^2 \le 4\int_\Omega \psi^2 |\nabla \phi|^2,\ \ \ \forall\,\phi\in C^1_0(\Omega).
\end{equation}
\end{proposition}

\begin{proof}[Proof of Proposition \ref{prop:p-Schwarz}]
Take a cut-off function $\chi:\mathbb R\rightarrow [0,1]$ such that $\chi|_{(-\infty,1/2]}=0$,  $\chi|_{[1,\infty)}=1$ and $\sup|\chi'|\le 2$.  Given $f\in A^p(\Omega)$,  set $\phi:=\chi(\delta/d) |f|^{p/2}$.  Then we have
\begin{eqnarray*}
\lambda_1(\Omega)\int_E |f|^p & \le & \lambda_1(\Omega)\int_\Omega \phi^2 \le  \int_\Omega |\nabla \phi|^2\\
& \le & \frac{8}{d^2} \int_{d/2\le \delta\le d} |f|^p+2\int_{\delta\ge d/2} \left|\nabla |f|^{p/2}\right|^2. 
\end{eqnarray*}
Since $|f|^{p/2}$ is (pluri)subharmonic on $\Omega$,  it follows from \eqref{eq:Caccioppoli} that
\begin{eqnarray*}
\int_{\delta\ge d/2} \left|\nabla |f|^{p/2}\right|^2 & \le & \int_{\Omega} \chi(2\delta/d)^2 \left|\nabla |f|^{p/2}\right|^2\le 4\int_\Omega |f|^p |\nabla \chi(2\delta/d)|^2\\
& \le & \frac{64}{d^{2}} \int_{d/4\le \delta\le d/2} |f|^p.
\end{eqnarray*}
Thus we have
\begin{eqnarray*}
\lambda_1(\Omega)\int_E |f|^p & \le &  \frac{8}{d^2} \int_{d/2\le \delta\le d} |f|^p + \frac{128}{d^{2}} \int_{d/4\le \delta\le d/2} |f|^p\\
& \le & \frac{128}{d^{2}}\int_{\Omega\backslash E} |f|^p,
\end{eqnarray*}
from which \eqref{eq:p-SchwarzUpper} immediately follows.
\end{proof}

\begin{proposition}\label{prop:p-Schwarz2}
Let $\Omega$ be a bounded  domain in $\mathbb C^n$ with $n\ge 2$ and $E$ a measurable relatively compact subset in $\Omega$.  Then for any $1<p<\infty$,
\begin{equation}\label{eq:p-SchwarzUpper3}
s_p(E,\Omega)\le \frac{C_{n,p}(\Omega)}{ C_{n,p}(\Omega)+ d^{p/ \tilde{p}}}
\end{equation}
where  $\tilde{p}$ is for the largest integer smaller than $p$,  and 
$$
C_{n,p}(\Omega)=\left(\frac{C_n|\Omega|^{\frac1{2n}}}{p/ \tilde{p}-1}\right)^{p/ \tilde{p}}.
$$
\end{proposition}

Since $p/ \tilde{p}\rightarrow 1+$ as $p\rightarrow \infty$,  it is natural to ask

\begin{problem}
Does there exist for every $\alpha>1$ a constant $C=C(n,p,\alpha,\Omega)>0$ such that
$$
s_p(E,\Omega)\le \frac{C}{C+d^\alpha}?
$$
\end{problem}

To prove Proposition \ref{prop:p-Schwarz2},  we need the following well-known consequence of the $L^p-$mapping property of the Beurling-Ahlfors operator.   

\begin{theorem}[cf.  \cite{Ahlfors}]\label{th:BA}
There exists a numerical constant $C_0\ge 1$ such that 
$$
\|\bar{\partial}\phi\|_{L^p(\mathbb C)} \le C_0(p^\ast-1)\|\partial \phi\|_{L^p(\mathbb C)},\ \ \ \forall\,\phi\in C^1_0(\mathbb C),
$$
where $1< p<\infty$ and $p^\ast=\max\{p,\frac1{p-1}\}$.
\end{theorem}

\begin{proposition}
Set $C_p^\ast=C_0 (p^\ast-1)$. Then for any $1\le j\le n$,
\begin{equation}\label{eq:dbarVSd}
\|\partial \phi/\partial \bar{z}_j\|_{L^p(\mathbb C^n)} \le C_p^\ast \|\partial \phi/\partial {z}_j\|_{L^p(\mathbb C^n)},\ \ \ \forall\,\phi\in C^1_0(\mathbb C^n).
\end{equation}
\end{proposition}

\begin{proof}
Since $\phi(\cdot,z')\in C^1_0(\mathbb C)$ for given $z':=(z_2,\cdots,z_n)$,  it follows from Theorem \ref{th:BA} that
$$
\int_{\mathbb C} |\partial \phi/\partial \bar{z}_1(\cdot,z')|^p \le (C_p^\ast)^p \int_{\mathbb C} |\partial \phi/\partial {z}_1(\cdot,z')|^p,
$$
so that 
$$
\int_{\mathbb C^n} |\partial \phi/\partial \bar{z}_1|^p\le (C_p^\ast)^p \int_{\mathbb C^n} |\partial \phi/\partial {z}_1|^p.
$$
Analogously,  we have 
$$
\int_{\mathbb C^n} |\partial \phi/\partial \bar{z}_j|^p\le (C_p^\ast)^p \int_{\mathbb C^n} |\partial \phi/\partial {z}_j|^p,\ \ \ 
j\ge 2.
$$
\end{proof}

Replace $\phi$ by $\bar{\phi}$ in \eqref{eq:dbarVSd},  we have
\begin{equation}\label{eq:dbarVSd2}
\|\partial \phi/\partial {z}_j\|_{L^p(\mathbb C^n)} \le C_p^\ast \|\partial \phi/\partial \bar{z}_j\|_{L^p(\mathbb C^n)},\ \ \ \forall\,1\le j\le n,\,\forall\,\phi\in C^1_0(\mathbb C^n).
\end{equation}
With $z_j=x_j+iy_j$,  we have 
$$
\partial \phi/\partial x_j=\partial\phi/\partial z_j+\partial\phi/\partial \bar{z}_j,\ \ \ \partial \phi/\partial y_j=i\partial\phi/\partial z_j-i\partial\phi/\partial \bar{z}_j.
$$
Thus the Minkowski inequality gives
\begin{equation}\label{eq:dVSdbar}
\| \nabla \phi\|_{L^p(\mathbb C^n)} \le 2(C_p^\ast+1) \sum_{j=1}^n \|\partial \phi/\partial\bar{z}_j\|_{L^p(\mathbb C^n)},\ \ \ \forall\,\phi\in C^1_0(\mathbb C^n).
\end{equation}
Recall the following Sobolev inequality:
\begin{equation}\label{eq:Sobolev}
\|\phi\|_{L^{\frac{2n}{2n-1}}(\mathbb C^n)}\le C_n \| \nabla \phi\|_{L^1(\mathbb C^n)},\ \ \ \forall\,\phi\in C^1_0(\mathbb C^n).
\end{equation}
Replace $\phi$ by $|\phi|^q$ with $q=\frac{2n-1}{2n-p}p$ into \eqref{eq:Sobolev},  we obtain
\begin{eqnarray*}
\left(\int_{\mathbb C^n} |\phi|^{\frac{2np}{2n-p}}\right)^{\frac{2n-1}{2n}} & \le & C_n q \int_{\mathbb C^n} |\phi|^{q-1}|\nabla \phi|\\
& \le & C_n q \| \nabla \phi\|_{L^p(\mathbb C^n)} \left(\int_{\mathbb C^n} |\phi|^{\frac{2np}{2n-p}}\right)^{1-\frac1p},
\end{eqnarray*}
i.e.,
$$
\|\phi\|_{L^{\frac{2np}{2n-p}}(\mathbb C^n)} \le C_n \frac{2n-1}{2n-p}p \| \nabla \phi\|_{L^p(\mathbb C^n)}.
$$
This combined with \eqref{eq:dVSdbar} gives
\begin{equation}\label{eq:dVSdbar2}
\|\phi\|_{L^{\frac{2np}{2n-p}}(\mathbb C^n)}\le \frac{C_n}{p-1}  \sum_{j=1}^n \|\partial \phi/\partial\bar{z}_j\|_{L^p(\mathbb C^n)},\ \ \ \forall\,\phi\in C^1_0(\mathbb C^n),
\end{equation}
when $1<p\le 2$ and $n>1$.  Here and in what follows we use the same symbol $C_n$ to denote all positive constants depending only on $n$.  

By \eqref{eq:dVSdbar2} and H\"older's inequality,  we have
\begin{equation}\label{eq:dVSdbar3}
\|\phi\|_{L^{p}(\mathbb C^n)}\le \frac{C_n|\Omega|^{\frac1{2n}}}{p-1}  \sum_{j=1}^n \|\partial \phi/\partial\bar{z}_j\|_{L^p(\mathbb C^n)},\ \ \ \forall\,\phi\in C^1_0(\mathbb C^n),
\end{equation}

\begin{proof}[Proof of Proposition \ref{prop:p-Schwarz2}]
  Take a cut-off function $\chi:\mathbb R\rightarrow [0,1]$ such that $\chi|_{(-\infty,1/2]}=0$,  $\chi|_{[1,\infty)}=1$ and $\sup|\chi'|\le 2$.  Given $f\in A^p(\Omega)$ with $1<p<\infty$,  set $\phi:=\chi(\delta/d) f^{\tilde{p}}$.  Since $1< p/\tilde{p}\le 2$,   \eqref{eq:dVSdbar3} implies   
\begin{eqnarray*}
\left(\int_E |f|^p\right)^{\tilde{p} /p}  \le  \|\phi\|_{L^{p/ \tilde{p}}(\Omega)}
& \le & \frac{C_n|\Omega|^{\frac1{2n}}}{p/ \tilde{p}-1}  \sum_{j=1}^n \|\partial \phi/\partial\bar{z}_j\|_{L^{p/ \tilde{p}}(\Omega)}\\
& \le & \frac{C_n|\Omega|^{\frac1{2n}}}{p/ \tilde{p}-1}\cdot \frac1d \left(\int_{\Omega\backslash E} |f|^p\right)^{\tilde{p} /p}, 
\end{eqnarray*}  
from which the assertion immediately follows.
\end{proof}

Proposition \ref{prop:p-Schwarz} and Proposition \ref{prop:p-Schwarz2} yield the following

\begin{corollary}\label{cor:p-SchwarzLower}
Let $\Omega$ be a bounded domain in $\mathbb C^n$.  Set $\Omega_\varepsilon:=\{z\in \Omega:\delta(z)>\varepsilon\}$.  Then
\begin{enumerate}
\item[$(1)$]  $1-s_p(\Omega_\varepsilon,\Omega)\gtrsim \varepsilon^2$ for\/ $0<p<\infty$;
\item[$(2)$]  $1-s_p(\Omega_\varepsilon,\Omega)\gtrsim \varepsilon^{p/\tilde{p}}$ for\/ $1<p<\infty$.
\end{enumerate}
Here the implicit constants depend only on $n,p,\Omega$.
\end{corollary}

On the other hand,  we have

\begin{proposition}\label{prop:p-SchwarzUpper}
Let $\Omega$ be a bounded domain in $\mathbb C^n$ with $\alpha(\Omega)>0$.  Then for any $1\le p\le 2$ and $\alpha<\alpha(\Omega)$ there is a constant $C>0$ such that $1-s_p(\Omega_\varepsilon,\Omega)\le C \varepsilon^{\alpha}$. 
\end{proposition}

The proof follows directly from \eqref{eq:B-Integ_1} and the following

\begin{lemma}
Let $\Omega$ be a bounded domain in $\mathbb C^n$.   Then
$$
\int_{\Omega\backslash \Omega_\varepsilon} |K_p(\cdot,z)|^p\ge \left[1-s_p(\Omega_\varepsilon,\Omega)\right]K_p(z)^{p-1},\ \ \ \forall\,z\in \Omega.
$$
\end{lemma}

\begin{proof}
For every $f\in A^p(\Omega)$,  we have
$$
\int_\Omega |f|^p -\int_{\Omega\backslash \Omega_\varepsilon} |f|^p \le s_p(\Omega_\varepsilon,\Omega)\int_\Omega |f|^p,
$$
i.e.,  
$$
\int_{\Omega\backslash \Omega_\varepsilon} |f|^p\ge \left[1-s_p(\Omega_\varepsilon,\Omega)\right]\int_\Omega |f|^p.
$$
It suffices to substitute $f=K_p(\cdot,z)$ into the previous inequality.  
\end{proof}

\begin{definition}
Let $\Omega$ be a bounded domain in $\mathbb C^n$.  The (one-sided) upper and lower $p-$Schwarz dimensions of $\partial \Omega$ are given by
\begin{eqnarray*}
\overline{\mathrm{dim}}_{p}(\partial \Omega) & := & 2n-\varliminf_{\varepsilon\rightarrow 0+} \frac{\log(1-s_p(\Omega_\varepsilon,\Omega))}{\log\varepsilon}\\
\underline{\mathrm{dim}}_{p}(\partial \Omega) & := & 2n-\varlimsup_{\varepsilon\rightarrow 0+} \frac{\log(1-s_p(\Omega_\varepsilon,\Omega))}{\log\varepsilon}
\end{eqnarray*}
respectively.  
\end{definition}

Corollary \ref{cor:p-SchwarzLower} and Proposition \ref{prop:p-SchwarzUpper} yield the following

\begin{corollary}
\begin{enumerate}
\item[$(1)$] $\underline{\mathrm{dim}}_{p}(\partial \Omega)\ge 2n-2$,   $\forall\,  0<p<\infty$;
\item[$(2)$] $\underline{\mathrm{dim}}_{p}(\partial \Omega)\ge 2n-p/\tilde{p}$,  $\forall\,1<p<\infty$;
\item[$(3)$] $\overline{\mathrm{dim}}_{p}(\partial \Omega)\le 2n-\alpha(\Omega)$,  $\forall\,1\le p\le 2$.
\end{enumerate}
\end{corollary}

Recall that the (one-sided) upper and lower Minkowski dimensions of $\partial \Omega$ are given by
\begin{eqnarray*}
\overline{\mathrm{dim}}_{\mathcal M}(\partial \Omega) & := & 2n-\varliminf_{\varepsilon\rightarrow 0+} \frac{\log(|\Omega|-|\Omega_\varepsilon|)}{\log\varepsilon}\\
\underline{\mathrm{dim}}_{\mathcal M}(\partial \Omega) & := & 2n-\varlimsup_{\varepsilon\rightarrow 0+} \frac{\log(|\Omega|-|\Omega_\varepsilon|)}{\log\varepsilon}
\end{eqnarray*}
respectively.  Since $s_p(\Omega_\varepsilon,\Omega)\ge |\Omega_\varepsilon|/|\Omega|$,  we immediately get the following

\begin{corollary}
$$
\overline{\mathrm{dim}}_{p}(\partial \Omega)\le \overline{\mathrm{dim}}_{\mathcal M}(\partial \Omega),\ \ \ \forall\,0<p<\infty. 
$$
\end{corollary} 

\begin{remark}
By the above corollaries,  we have $\overline{\mathrm{dim}}_{\mathcal M}(\partial \Omega)\ge 2n-p/\tilde{p}$.  Letting $p\rightarrow \infty$,  we obtain $\overline{\mathrm{dim}}_{\mathcal M}(\partial \Omega)\ge 2n-1$.
\end{remark}

\subsection{Non-uniqueness of best approximations}

\begin{lemma}\label{lm:minimal_element_p_Schwarz}
Let $\Omega$ be a domain in $\mathbb{C}^n$ and $E$ a measurable relatively compact subset in $\Omega$. If $s_p(E,\Omega)>0$, then there exists $f_E\in{A^p(\Omega)}\setminus\{0\}$ such that
\begin{equation}\label{eq:NU_1}
s_p(E,\Omega)=\frac{\int_E|f_E|^p}{\int_\Omega|f_E|^p}.
\end{equation}
\end{lemma}

\begin{proof}
Take a sequence $\{f_j\}\subset {A^p(\Omega)\setminus\{0\}}$ such that $\int_\Omega|f_j|^p=1$ and 
$$
\int_E|f_j|^p>(1-1/j)s_p(E,\Omega).
$$
Clearly, $\{f_j\}$ forms a normal family and by extracting a subsequence we may assume that $f_j$ converges locally uniformly to certain $f_E\in {\mathcal O(\Omega)}$ and Fatou's lemma gives
\[
\int_\Omega|f_E|^p\leq\liminf_{j\rightarrow\infty}\int_\Omega|f_j|^p=1.\qedhere
\]
Since $E$ is relatively compact in $\Omega$, we have
\[
\int_E|f_E|^p=\lim_{j\rightarrow\infty}\int_E|f_j|^p=s_p(E,\Omega)>0.
\]
Thus
$$
s_p(E,\Omega)\ge \frac{\int_E|f_E|^p}{\int_\Omega|f_E|^p}\ge s_p(E,\Omega),
$$
i.e.,  \eqref{eq:NU_1} holds.
\end{proof}

\begin{lemma}\label{lm:continuous_Schwarz}
For any bounded domain $\Omega$ in $\mathbb{C}^n$, there exists an relatively compact open subset $E$ in $\Omega$ such that $s_p(E,\Omega)=1/2$.
\end{lemma}

\begin{proof}
By  property $(e)$ in \S\,7.1,  there exists a compact subset $K\subset\Omega$ such that
\[
s_p(K,\Omega)\geq\frac{|K|}{|\Omega|}>\frac{1}{2}.
\]
For fixed $0<\delta_0<d(K,\partial\Omega)$, there are a finite number of  balls $B(z_1,\delta_0),\cdots,B(z_N,\delta_0)$ such that
\[
K\subset\bigcup^N_{k=1}B(z_k,\delta_0).
\]
We define
\[
\Omega_t:=\bigcup^N_{k=1}B(z_k,t),\ \ \ 0<t<\delta_0.
\]
Clearly, $K\subset\Omega_{\delta_0}$ and $|\Omega_t|\rightarrow0$ when $t\rightarrow0$. Moreover, for $0<t_1<t_2<\delta_0$, we have
\begin{equation}\label{eq:volume_difference}
|\Omega_{t_2}\setminus\Omega_{t_1}|\leq\sum^N_{k=1}|B(z_k,t_2)\setminus{B(z_k,t_1)}|=N|\mathbb{B}^n|(t_2^{2n}-t_1^{2n}),
\end{equation}
where $\mathbb{B}^n$ is the unit ball in $\mathbb{C}^n$.

We claim that
\[
t\mapsto{s_p(\Omega_t,\Omega)},\ \ \ 0<t\leq\delta_0
\]
is a continuous function. First, we show the left continuity at $t_0\in(0,\delta_0]$.  By Lemma \ref{lm:minimal_element_p_Schwarz}, there exists $f_{t_0}\in{A^p(\Omega)\setminus\{0\}}$ such that $s_p(\Omega_{t_0},\Omega)=\int_{\Omega_{t_0}}|f_{t_0}|^p/\int_\Omega|f_{t_0}|^p$. Hence for $0<t<t_0$,  
\begin{eqnarray*}
s_p(\Omega_{t_0},\Omega) & \geq & {s_p(\Omega_t,\Omega)}\  \geq\   \frac{\int_{\Omega_t}|f_{t_0}|^p}{\int_\Omega|f_{t_0}|^p}\\
&=& \frac{\int_{\Omega_{t_0}}|f_{t_0}|^p-\int_{\Omega_{t_0}\setminus\Omega_t}|f_{t_0}|^p}{\int_\Omega|f_{t_0}|^p}\\
&=& s_p(\Omega_{t_0},\Omega)-\frac{\int_{\Omega_{t_0}\setminus\Omega_t}|f_{t_0}|^p}{\int_\Omega|f_{t_0}|^p}.
\end{eqnarray*}
The Bergman inequality asserts that
\begin{equation}\label{eq:Bergman_ineq_Chebyshev}
|f(z)|^p\leq{C_{\delta_0}}\int_\Omega|f|^p,\ \ \ z\in\Omega_{\delta_0},\ f\in{A^p(\Omega)}.
\end{equation}
Thus
\[
s_p(\Omega_{t_0},\Omega)\geq{s_p(\Omega_t,\Omega)}\geq{s_p(\Omega_{t_0},\Omega)}-C_{\delta_0}|\Omega_{t_0}\setminus\Omega_t|,
\]
which together with \eqref{eq:volume_difference} imply $\lim_{t\rightarrow{t_0-}}s_p(\Omega_t,\Omega)=s_p(\Omega_{t_0},\Omega)$. 

On the other hand, for $0<t_0<t<\delta_0$ and $f\in{A^p(\Omega)\setminus\{0\}}$,  one has
\[
\frac{\int_{\Omega_t}|f|^p}{\int_\Omega|f|^p}=\frac{\int_{\Omega_{t_0}}|f|^p+\int_{\Omega_t\setminus\Omega_{t_0}}|f|^p}{\int_\Omega|f|^p}\leq{s_p(\Omega_{t_0},\Omega)}+\frac{\int_{\Omega_t\setminus\Omega_{t_0}}|f|^p}{\int_\Omega|f|^p}.
\]
Again, the Bergman inequality \eqref{eq:Bergman_ineq_Chebyshev} gives
\[
\frac{\int_{\Omega_t}|f|^p}{\int_\Omega|f|^p}\leq{s_p(\Omega_{t_0},\Omega)}+C_{\delta_0}|\Omega_t\setminus\Omega_{t_0}|.
\]
Since $f\in{A^p(\Omega)\setminus\{0\}}$ can be arbitrarily chosen,  we have
\[
s_p(\Omega_{t_0},\Omega)\leq{s_p(\Omega_t,\Omega)}\leq{s_p(\Omega_{t_0},\Omega)}+C'_{t_0}|\Omega_t\setminus\Omega_{t_0}|,
\]
so that
 $\lim_{t\rightarrow{t_0+}}s_p(\Omega_t,\Omega)=s_p(\Omega_{t_0},\Omega)$ in view of \eqref{eq:volume_difference}.
 
Note that $s_p(\Omega_{\delta_0},\Omega)\ge s_p(K,\Omega)>1/2$,  and property $(f)$ in \S\,7.1 yields  $\lim_{t\rightarrow0+}s_p(\Omega_t,\Omega)=0$.  By the intermediate value theorem, there exists some $\Omega_{t_0}=:E$ with $s_p(E,\Omega)=1/2$.
\end{proof}

\begin{proof}[Proof of Theorem \ref{th:not_Chebyshev}]
By Lemma \ref{lm:minimal_element_p_Schwarz} and Lemma \ref{lm:continuous_Schwarz}, there exists a relative compact open subset $E$ in $\Omega$ and $f_E\in{A^p(\Omega)\setminus\{0\}}$ such that
\[
\frac{\int_E|f_E|^p}{\int_\Omega|f_E|^p}=s_p(E,\Omega)=\frac{1}{2}.
\]
Hence
\begin{equation}\label{eq:not_Chebyshev1}
\int_E|f_E|^p=\int_{\Omega\setminus{E}}|f_E|^p=:I.
\end{equation}
Since 
$$
\frac12=s_p(E,\Omega) \ge \frac{\int_E|f|^p}{\int_\Omega |f|^p}
= \frac{\int_E|f|^p} {\int_E|f|^p+ \int_{\Omega\setminus {E}} |f|^p}
$$
for all $f\in A^p(\Omega)\setminus\{0\}$,  it follows that
\begin{equation}\label{eq:not_Chebyshev2}
\int_E|f|^p\leq\int_{\Omega\setminus{E}}|f|^p,\ \ \ \forall\,f\in A^p(\Omega).
\end{equation}

Given $h_1\in{A^p(\Omega)}$,    define $h_2:=h_1+f_E\in{A^p(\Omega)}$ and
\[
g(z):=\begin{cases}
h_1(z),\ \ \ z\in{E},\\
h_2(z),\ \ \ z\in\Omega\setminus{E}.
\end{cases}
\]
Clearly, $g\in{L^p(\Omega)}$.  Moreover,  since $|a-b|^p\leq|a|^p+|b|^p$ for all $a,b\in\mathbb{C}$ and $0<p\leq1$, we infer from \eqref{eq:not_Chebyshev2} that for $h\in{A^p(\Omega)}$,  
\begin{eqnarray*}
\int_\Omega|g-h|^p & = & \int_E|h_1-h|^p+\int_{\Omega\setminus{E}}|h_2-h|^p\\
& \geq & \int_E|h_1-h|^p+\int_E|h_2-h|^p\\
& \geq & \int_E|h_1-h_2|^p
\ = \ I.
\end{eqnarray*}
On the other hand,   \eqref{eq:not_Chebyshev1} implies
\[
\int_\Omega|g-h_1|^p=\int_{\Omega\setminus{E}}|h_2-h_1|^p=\int_{\Omega\setminus{E}}|f_E|^p=I
\]
and
\[
\int_\Omega|g-h_2|^p=\int_E|h_1-h_2|^p=\int_E|f_E|^p=I.
\]
Thus $h_1,h_2\in\mathcal{P}_{A^p(\Omega)}(g)$. Clearly,  $h_1\neq h_2$.
\end{proof}

\section{Proofs of Theorem \ref{th:minimum_of_Bergman_kernel} and Proposition \ref{prop:Non-rigid}} 

We first give a simple fact as follows. 

\begin{lemma}\label{lm:mean_value}
Let $\Omega$ be a bounded domain in $\mathbb{C}^n$ and $1\le p<\infty$.  Then $K_p(a)=1/|\Omega|$\/ for some $a\in\Omega$ if and only if
\begin{equation}\label{eq:mean_value}
f(a)=\frac{1}{|\Omega|}\int_\Omega{f},\ \ \ \forall\,f\in{A^p(\Omega)}.
\end{equation}
\end{lemma}

\begin{proof}
Only if part.  Since
$$
m_p(a)=K_p(a)^{-1/p}=|\Omega|^{1/p}=\left(\int_\Omega 1^p \right)^{1/p},
$$
it follows from the uniqueness of the minimizer in \eqref{eq:MinProb} that $m_p(\cdot,a)=1$,  i.e.,  $K_p(\cdot,a)=K_p(a)=1/|\Omega|$.  This together with  the reproducing formula \eqref{eq:RPF} yield \eqref{eq:mean_value}.

If part.  Note that 
$$
K_p(a)=\frac1{|\Omega|} \int_\Omega K_p(\cdot,a) \le \frac1{|\Omega|} |\Omega|^{1-\frac1p} \left(\int_\Omega |K_p(\cdot,a)|^p \right)^{\frac1p}=|\Omega|^{-\frac1p} K_p(a)^{1-\frac1p},
$$
i.e.,  $K_p(a)\le 1/|\Omega|$.  On the other hand,  we naturally have $K_p(a)\ge 1/|\Omega|$.  Thus $K_p(a)=1/|\Omega|$.  
\end{proof}

\begin{proof}[Proof of Theorem \ref{th:minimum_of_Bergman_kernel}/$(1)$]
We first deal with the case $1\le p <2$.  Without loss of generality,  we may assume that 
$$
a=0, \,\delta(a):=d(a,\partial \Omega)=1\ \text{and}\  1\in\partial\Omega.
$$
 Clearly,  $\mathbb{D}\subset \Omega$,  thus it suffices to show that $\Omega\subset\mathbb{D}$.  Set
\[
h(z)=\frac{z+1}{z-1}+1,\ \ \ z\in\Omega.
\]
It is easy to verify that  $h\in{A^p(\Omega)}$ for all  $1\leq{p}<2$;  moreover,  
\begin{equation}\label{eq:real_part_minimum}
\mathrm{Re}\,h(z)=\frac{|z|^2-1}{|z-1|^2}+1,
\end{equation}
 is an integrable harmonic function on $\Omega$,  so that 
 \[
\int_{\mathbb{D}}\mathrm{Re}\,h=0,
\]
in view of the mean-value property.  On the other hand,  
  Lemma \ref{lm:mean_value} implies
\[
\int_\Omega\mathrm{Re}\,h=\mathrm{Re}\,\int_\Omega{h}=|\Omega|\cdot\mathrm{Re}\,h(0)=0.
\]
Thus
\begin{equation}\label{eq:interal_zero}
\int_{\Omega\setminus\mathbb{D}}\mathrm{Re}\,h=0
\end{equation}
Since  $\mathrm{Re}\,h\geq 1$ holds on $\Omega\setminus\mathbb{D}$ in view of \eqref{eq:real_part_minimum}, it follows from  \eqref{eq:interal_zero}  that $\Omega\setminus{\mathbb D}$ is of zero measure,  which has to be empty,  for  $\Omega$ is a domain.  Thus  $\Omega\subset\mathbb{D}$.

Now suppose $0<p<1$.  By (6.3) in \cite{CZ},  we have
$$
1=|\Omega|^{\frac1p}\cdot K_p(a)^{\frac1p}\ge |\Omega|\cdot K_1(a)\ge 1,
$$
so that $K_1(a)=1/|\Omega|$.  By the previous argument,  we conclude that $\Omega$ is a disc centred at $a$.   
\end{proof}

To complete the proof  of Theorem \ref{th:minimum_of_Bergman_kernel},  we need the following

\begin{theorem}[cf. \cite{HedbergApproximation}, Theorem 1]\label{th:Hedberg_approximation}
Let $E$ be a bounded measurable set in $\mathbb{C}$ and $1\leq{q}<2$. Then for every function $f\in{L^q(E)}\cap\mathcal{O}(E^\circ)$, there exists a sequence of rational functions $\{f_k\}$ with poles in $\mathbb{C}\setminus{E}$ such that
\[
\lim_{k\rightarrow\infty}\|f_k-f\|_{L^q(E)}=0.
\]
\end{theorem}

\begin{proof}[Proof of Theorem \ref{th:minimum_of_Bergman_kernel}/$(2)$]
Let $p>2$ and set $\Omega_1=\overline{\Omega}^\circ$.  Apply Theorem \ref{th:Hedberg_approximation} with $q=1$ and $E=\overline{\Omega}$,  we see that every $f\in{A^1(\Omega_1)}$ can be approximated in the $L^1$-norm by a sequence of rational functions $\{f_k\}$ with poles in $\mathbb{C}\setminus\overline{\Omega}$.  In paricular,  $f_k\in{A^p(\Omega)}$, $f_k$ converges locally uniformly to $f$ and
\[
\lim_{k\rightarrow\infty}\int_\Omega{f_k}=\int_\Omega{f}.
\]
Thus it follows from Lemma \ref{lm:mean_value} that
\[
f(a)=\lim_{k\rightarrow\infty}f_k(a)=\lim_{k\rightarrow\infty}\frac{1}{|\Omega|}\int_\Omega{f_k}=\frac{1}{|\Omega|}\int_\Omega{f}=\frac{1}{|\Omega_1|}\int_{\Omega_1}f,
\]
for  $|\Omega_1\setminus\Omega|=0$.  The proof of Theorem \ref{th:minimum_of_Bergman_kernel}/$(1)$ yields that  $\Omega_1$ is a disc.
\end{proof}

\begin{remark}
In particular,  if $\Omega$ is fat,  then the condition $K_p(a)=1/|\Omega|$ for some $a\in\Omega$ and some $2\le p<\infty$ implies that $\Omega$ is a disc.
\end{remark}

\begin{corollary}
Let $\Omega=\Omega_1\times \cdots \times \Omega_n$,  where $\Omega_j$,  $1\le j\le n$,   are bounded planar domains.   Then the following properties hold:
\begin{enumerate}
\item[$(1)$] If $K_p(a)=1/|\Omega|$ for some $a\in\Omega$ and some $0<p<2$, then $\Omega$ is a polydisc centred at $a$.
\item[$(2)$] Suppose furthermore that all $\Omega_j$ are fat.  
If $K_p(a)=1/|\Omega|$ for some $a\in\Omega$ and some $2\le p<\infty$,  then $\Omega$ is a  polydisc centred at $a$.
\end{enumerate}
\end{corollary}

\begin{proof}
$(1)$ By the product rule of the $p-$Bergman kernel (cf.  \cite{CZ},  Proposition 2.8),  we have
$$
\frac1{|\Omega_1|}\cdots \frac1{|\Omega_n|}=\frac1{|\Omega|} = K_{\Omega,p}(a) = K_{\Omega_1,p}(a_1)\cdots K_{\Omega_n,p}(a_n),
$$
i.e.,  
$$
\prod_{j=1}^n \left[ |\Omega_j| \cdot K_{\Omega_j,p}(a_j) \right] =1.
$$
Since $|\Omega_j| \cdot K_{\Omega_j,p}(a_j)\ge 1$ for all $j$,  we have $|\Omega_j| \cdot K_{\Omega_j,p}(a_j)= 1$,  so that $\Omega_j$ has to be a disc centred at $a_j$ for every $j$, i.e.,  $\Omega$ is a polydisc centred at $a$.  

$(2)$  The argument is similar.   
\end{proof}

If $\Omega$ is a bounded complete circular domain in $\mathbb{C}^n$,  then $K_{\Omega,2}(0)=1/|\Omega|$ (cf.  \cite[\S\,4]{Boas}). The same conclusion actually holds for general $p$.  To see this,  it suffices to verify that if $\varphi$ is psh on $\Omega$,  then the following mean-value inequality holds: 
\begin{equation}\label{eq:mean_value_circular}
\varphi(0)\leq\frac{1}{|\Omega|}\int_\Omega\varphi,
\end{equation}
for $|f(0)|^p\leq\frac{1}{|\Omega|}\int_\Omega |f|^p$ would imply $K_p(0)\le 1/|\Omega|$.  
To get \eqref{eq:mean_value_circular},   first apply the coarea formula (cf. \cite[\S\,13.5]{Chirka}, see also \cite[Theorem 3.2.22]{Federer}) to the natural map $F:\mathbb{C}^n\rightarrow\mathbb{P}^{n-1}$,  that is,
\begin{equation}\label{eq:coarea}
\int_\Omega\varphi=\int_{\mathbb{P}^{n-1}}\int_{F^{-1}(\xi)}\varphi(z)|z|^{2n-2}d\sigma(z)dV_{FS}(\xi),
\end{equation}
where $F^{-1}(\xi)$ is the intersection of $\Omega$ with the complex line $\xi$, $d\sigma$ is the area measure on the complex line $\xi$, and $dV_{FS}$ denotes  the Fubini-Study volume element  on $\mathbb{P}^{n-1}$.   Since $\Omega$ is complete circular, $F^{-1}(\xi)$ can be identified with a disc in $\mathbb{C}$,  so that  the mean-value inequality implies
\[
\int_{F^{-1}(\xi)}\varphi(z)|z|^{2n-2}d\sigma(z)\geq{C(\xi)}\varphi(0),
\]
where
\[
C(\xi)=\int_{F^{-1}(\xi)}|z|^{2n-2}d\sigma(z).
\]
Thus
\[
\int_\Omega\varphi\geq\varphi(0)\int_{\mathbb{P}^{n-1}}C(\xi)dV_{FS}(\xi).
\]
Apply  \eqref{eq:coarea} with  $\varphi\equiv1$, we have
\[
\int_{\mathbb{P}^{n-1}}C(\xi)dV_{FS}(\xi)=|\Omega|,
\]
which in turn gives \eqref{eq:mean_value_circular}.

We are left to give the example in Proposition \ref{prop:Non-rigid}.  
Let $\varphi$ be a bounded continuous function on $\mathbb{D}$. Consider the following Hartogs domain
\[
\Omega_\varphi:=\left\{(z_1,z_2):\ |z_2|^2<e^{-\varphi(z_1)},\ z_1\in\mathbb{D}\right\}.
\]
By Ligocka's formula (cf. \cite{Ligocka}), we have
\[
K_{\Omega_\varphi,2}(0)=\frac{1}{\pi}K_{\mathbb{D},\varphi}(0),
\]
where $K_{\mathbb{D},\varphi}$ denotes the reproducing kernel of the weighted $L^2$ Bergman space $A^2(\mathbb{D},\varphi)$.  Since
\[
|\Omega_\varphi|=\pi\int_{\mathbb{D}}e^{-\varphi},
\]
we conclude that $K_{\Omega_\varphi,2}(0)=1/|\Omega_\varphi|$ if and only if
\begin{equation}\label{eq:Hartogs}
K_{\mathbb{D},\varphi}(0)=\frac{1}{\int_{\mathbb{D}}e^{-\varphi}}.
\end{equation}
Analogously to Lemma \ref{lm:mean_value},  one may show that \eqref{eq:Hartogs} is  equivalent to
$$
f(0)=\frac{\int_{\mathbb{D}}{f}e^{-\varphi}}{\int_{\mathbb{D}}{e^{-\varphi}}},\ \ \ \forall\,f\in{A^2(\mathbb D,\varphi)}.
$$
Since $\varphi$ is bounded,  this is equivalent to 
\begin{equation}\label{eq:mean_value_weighted}
f(0)=\frac{\int_{\mathbb{D}}{f}e^{-\varphi}}{\int_{\mathbb{D}}{e^{-\varphi}}},\ \ \ \forall\,f\in{A^2(\mathbb D)}.
\end{equation}
Note that \eqref{eq:mean_value_weighted} holds whenever $\varphi$ is radial, i.e., $\varphi(z)=\varphi(|z|)$. We are going to construct a non-radial weight $\varphi$ which still verifies  
 \eqref{eq:mean_value_weighted}. 
 
First,  fix  a smooth radial function $\psi$ on $\overline{\mathbb{D}}$ such that $e^{-\psi(0)}> \frac1\pi \int_\mathbb D e^{-\psi}$.  Then the real-valued function
\[
\eta:=e^{-\psi}-\frac{1}{\pi}\int_{\mathbb{D}}e^{-\psi}
\]
verifies
\begin{equation}\label{eq:orthogonal}
\int_{\mathbb{D}}{f\eta}=0,\ \ \ \forall\,f\in{A^2(\mathbb{D})},
\end{equation}
i.e., $\eta\in{A^2(\mathbb{\mathbb{D}})^\perp}\subset{L^2(\mathbb{D})}$.   Given $a\in\mathbb{D}$,  take $F_a(w):=(w-a)/(1-\bar{a}w)$.   It follows that $f\circ{F_a}\cdot{F_a'}\in{A^2(\mathbb{D})}$ for $f\in A^2(\mathbb D)$.  Moreover, since $F_a'(w)=(1-|a|^2)/(1-\bar{a}w)^2$ is bounded on $\mathbb{D}$, we see that $f\circ{F_a}\in{A^2(\mathbb{D})}$. Thus \eqref{eq:orthogonal} gives
\[
\int_{\mathbb{D}}f\circ{F_a}\cdot{\eta}=\int_{\mathbb{D}}f\cdot ({\eta\circ{F_a^{-1}}})  |(F_a^{-1})'|^2=0,\ \ \ \forall\,f\in{A^2(\mathbb{D})},
\]
that is,
\begin{equation}\label{eq:orthogonal2}
\widetilde{\eta}:=(\eta\circ{F_a^{-1}}) |(F_a^{-1})'|^2\in{A^2(\mathbb{D})^\perp}.
\end{equation}
Note that $F_a^{-1}(z)=(z+a)/(1+\bar{a}z)$ and $(F_a^{-1})'(z)=(1-|a|^2)/(1+\bar{a}z)^2$. Thus
\[
\widetilde{\eta}(a)=\eta\left(\frac{2a}{1+|a|^2}\right)\frac{(1-|a|^2)^2}{(1+|a|^2)^4},\ \ \ \widetilde{\eta}(-a)=\frac{\eta(0)}{(1-|a|^2)^2}.
\]
Letting $|a|\rightarrow1$, we have $\widetilde{\eta}(a)\rightarrow0$ and $\widetilde{\eta}(-a)\rightarrow+\infty$,  for $\eta$ is a bounded function with  $\eta(0) > 0$.  Thus there exists $a\in\mathbb{D}$ such that $\widetilde{\eta}(a)\neq\widetilde{\eta}(-a)$,  i.e.,   $\widetilde{\eta}$ is not a radial function.  Moreover, since $\widetilde{\eta}$ is a bounded on $\mathbb{D}$,  there exists $0<\varepsilon\ll1$ such that
\[
e^{-\psi}+\varepsilon\widetilde{\eta}>0\ \ \ \text{on}\ \ \overline{\mathbb D}.
\]
Set $\varphi:=-\log(e^{-\psi}+\varepsilon\widetilde{\eta})$.  From  \eqref{eq:orthogonal2} we know that
\[
\frac{\int_{\mathbb{D}}fe^{-\varphi}}{\int_{\mathbb{D}}e^{-\varphi}}
=\frac{\int_{\mathbb{D}}fe^{-\psi}-\varepsilon\int_{\mathbb{D}}f\widetilde{\eta}}{\int_{\mathbb{D}}e^{-\psi}-\varepsilon\int_{\mathbb{D}}\widetilde{\eta}}
=\frac{\int_{\mathbb{D}}fe^{-\psi}}{\int_{\mathbb{D}}e^{-\psi}}=f(0),\ \ \ \forall\,f\in{A^2(\mathbb{D})},
\]
which implies
 $$
 K_{\Omega_\varphi,2}(0)=1/|\Omega_\varphi|.
 $$
For $2<p<\infty$,  we have
 $$
 1= |\Omega_\varphi|^{\frac12} K_{\Omega_\varphi,2}(0)^{\frac12}\ge  |\Omega_\varphi|^{\frac1p} K_{\Omega_\varphi,p}(0)^{\frac1p}\ge 1,
 $$
so that $K_{\Omega_\varphi,p}(0)=1/|\Omega_\varphi|.$ 
 
As  $\varphi$ is smooth on $\overline{\mathbb{D}}$,  we see that $\Omega_\varphi$ has piecewise smooth boundary ($\partial \Omega_\varphi$ is  non-smooth only when $|z_1|=1$ and $|z_2|^2=e^{-\varphi(z_1)}$).

 $\Omega_\varphi$ is non-circular.  To see this,  first take  $z\in\mathbb{D}$ and $\lambda\in\partial\mathbb{D}$ such that $e^{-\varphi(z)}>e^{-\varphi(\lambda z)}$,  then take $r>0$ with
  $$
  e^{-\varphi(z)}>r^2>e^{-\varphi(\lambda{z})}.
  $$
  It follows that $(z,r)\in \Omega_\varphi$,  while $(\lambda{z},\lambda{r})\notin \Omega_\varphi$.

\begin{problem}
Does there exists a bounded non-circular fat domain $\Omega\ni 0$ on which $K_p(0)=1/|\Omega|$ holds for all $0<p<2$? 
\end{problem}

\section{Proof of Theorem \ref{th:completeness_isolated}}

We shall make use of the following two results,  whose proofs are the same as the standard cases.

\begin{proposition}[Localization Principle]\label{prop:localization}
Let $\Omega$ be a bounded pseudoconvex domain in $\mathbb C^n$ and $V\subset\subset U$ two given neighborhoods of $w\in \partial \Omega$.  Then there exist constants $C_1,C_2>0$ such that for all $p>0$,  
$$
C_1 ds^2_{\Omega,p}\le ds^2_{\Omega\cap U,p} \le C_2 ds^2_{\Omega,p}\ \ \ \text{on}\ \ \ \Omega\cap V.
$$
\end{proposition} 

\begin{proposition}[Kobayashi's Criterion]\label{prop:b_4}
Let $\Omega$ be a bounded domain in $\mathbb C^n$. Suppose there exists a dense subset $\mathcal S$ in $A^2(\Omega,\log K_p)$ such that for every $f\in \mathcal S$ and every discrete sequence $\{z_k\}$ in $\Omega$ there is a subsequence $\{z_{k_j}\}$ such that 
$$
\frac{|f(z_{k_j})|^2}{K_{2,p}(z_{k_j})}\rightarrow 0\ \ \ (j\rightarrow \infty).
$$
Then $ds^2_p$ is complete.
\end{proposition}

\begin{corollary}\label{cor:b_5}
Let $\Omega$ be a bounded domain in $\mathbb C^n$ which satisfies 
\begin{enumerate}
\item[$(1)$] $K_{2,p}(z)\rightarrow \infty$ as $z\rightarrow \partial \Omega$;
\item[$(2)$] for every $w\in \partial\Omega$, $A^\infty(\Omega,w)$ lies dense in $A^2(\Omega,\log K_p)$, where 
$$
A^\infty (\Omega,w):=\left\{f\in \mathcal O(\Omega): \limsup_{z\rightarrow w} |f(z)|<\infty\right\}.
$$
\end{enumerate}
Then $ds^2_p$ is complete.
\end{corollary}

\begin{proof}
Let $\{z_k\}$ be a discrete sequence in $\Omega$. We may choose a subsequence $z_{k_j}\rightarrow z_0\in \partial \Omega$. For every $f\in A^2_p(\Omega)$ there exists a sequence $\{f_m \}\subset A^\infty(\Omega,z_0)$ such that 
$$
\int_\Omega \frac{|f_m-f|^2}{K_p}\rightarrow 0\ \ \ (m\rightarrow \infty).
$$
Since
\begin{eqnarray*}
\frac{|f(z_{k_j})|^2}{K_{2,p}(z_{k_j})} & \le & \frac{2|f(z_{k_j})-f_m(z_{k_j})|^2}{K_{2,p}(z_{k_j})} + \frac{2|f_m(z_{k_j})|^2}{K_{2,p}(z_{k_j})}\\
& \le & 2  \int_\Omega \frac{|f_m-f|^2}{K_p} + \frac{2|f_m(z_{k_j})|^2}{K_{2,p}(z_{k_j})},
\end{eqnarray*}
it follows from (1) that 
$$
\limsup_{j\rightarrow \infty} \frac{|f(z_{k_j})|^2}{K_{2,p}(z_{k_j})} \le 2  \int_\Omega \frac{|f_m-f|^2}{K_p}\rightarrow 0\ \ \ (m\rightarrow \infty),
$$
and Kobayashi's Criterion applies.
\end{proof}

It turns out that condition (1) in Corollary \ref{cor:b_5} is superfluous  for $p<2$ at least for the case $n=1$:

\begin{lemma}\label{lm:b_6}
Let $\Omega$ be a bounded domain in $\mathbb C$ and $0<p<2$. Then  
\begin{equation}\label{eq:p-Lower_1}
 \frac{2-p}{2\pi R^{2-p}} \delta(z)^{-p} \le K_p(z)\le \pi^{-1}\delta(z)^{-2},
 \end{equation}
 \begin{equation}\label{eq:p-Lower_2}
K_{2,p}(z)\ge C_0 \left(\frac{2-p}{2\pi}\right) R^{p-4}\cdot \delta(z)^{-p},
\end{equation}
where $C_0$ is a universal constant and  $R$ means the diameter of $\Omega$.
\end{lemma}

\begin{proof}
Given $z\in \Omega$,   take $z_\ast \in \partial \Omega$ so that $|z-z_\ast|=\delta(z)$.  Set $f_\ast(\zeta)=(\zeta-z_\ast)^{-1}$.  Since 
$$
\int_\Omega |f_\ast|^p\le \frac{2\pi R^{2-p}}{2-p}, 
$$
 we see that 
$$
K_p(z)\ge \frac{|f_\ast(z)|^p}{\int_\Omega |f_\ast|^p}\ge \frac{2-p}{2\pi R^{2-p}}\cdot \delta(z)^{-p}. 
$$
The second inequality in \eqref{eq:p-Lower_1} follows directly from the Bergman inequality.   

Since $\log K_p$ is subharmonic,  it follows from  Demailly's approximation \cite{Demailly} that 
$$
K_{2,p}(z)\ge \frac{C_0}{R^2} \cdot K_p(z).
$$
The proof is complete.
\end{proof}

Let $\Omega$ be a bounded  domain in $\mathbb C$.  Let $w$ be a non-isolated boundary point, i.e., there exists a sequence $\{w_j\}\subset \mathbb C\setminus \Omega$ such that $w_j\neq w$ for all $j$ and $w_j\rightarrow w$ as $j\rightarrow \infty$.  Set
$$
\Omega_{w}^{j}:=\Omega\cup \{z:|z-w|<|w_j-w|\}, \ \ \ j\in \mathbb Z^+.
$$ 
Without loss of generality, we may assume that $\{|w_j-w|\}$ is strictly decreasing, so that  $\Omega^j_w$ is also decreasing. 
Set $\delta_j:=\delta_{\Omega^j_w}$.  Since $w_j\in \partial \Omega_{w}^{j}$,  we have 
\begin{equation}\label{eq:b-dist_1}
\delta_j(z)\le |z-w|+\delta_j(w)\le 2|w_j-w|
\end{equation}
for all $z\in \Omega^j_w\backslash \Omega$, and
\begin{equation}\label{eq:b-dist_2}
\delta_j(z)\rightarrow \delta(z)\ \ \ \text{uniformly\ on}\ \ \ \overline{\Omega}. 
\end{equation}
 
\begin{lemma}\label{lm:M_2}
If\/ $1\le p<2$, then $K_{\Omega_{w}^{j}, p}(z) \uparrow K_{\Omega,p}(z)$ for $z\in \Omega$ as $j\rightarrow \infty$.
\end{lemma}

\begin{proof}
Given $z\in \Omega$,  choose $f\in A^p(\Omega)$ such that $\|f\|_{L^p(\Omega)}=1$ and $|f(z)|^p=K_{\Omega, p}(z)$. 
Set $E=\Omega\cup \{w\}$.  Since $w$ is a non-isolated boundary point,   we see that $E^\circ =\Omega$. Apply Theorem \ref{th:Hedberg_approximation} to the set $E$, we get  a sequence 
$\{f_k\}\subset \mathcal O({\Omega})\cap \mathcal O_w$, where $\mathcal O_w$ is the germ of  holomorphic functions at $w$, such that 
$$
\|f_k-f\|_{L^p({\Omega})}\rightarrow 0\ \ \ (k\rightarrow \infty).
$$
We also have 
\begin{equation}\label{eq:uniform}
|f_k(z)-f(z)|^p \le \frac1{\pi\delta(z)^2} \|f_k-f\|_{L^p({\Omega})}^p\rightarrow 0
\end{equation}
in view of the mean-value inequality. 
Clearly, for every $k$ there exists $j_k\gg 1$ such that $f_k\in \mathcal O(\Omega^{j_k}_w)$ and 
$\|f_k\|_{L^p(\Omega^{j_k}_w\setminus {\Omega})}\rightarrow 0$ as $k\rightarrow \infty$. Thus 
$$
\|f_k\|_{L^p(\Omega^{j_k}_w)}\rightarrow \|f\|_{L^p(\Omega)}=1 \ \ \  (k\rightarrow \infty).
$$
It follows that 
$$
\varliminf_{k\rightarrow \infty} K_{\Omega^{j_k}_w,p}(z)\ge \varliminf_{k\rightarrow \infty}  \frac{|f_k(z)|^p}{\|f_k\|_{L^p(\Omega^{j_k}_w)}^p}
=|f(z)|^p=K_{\Omega,p}(z).
$$
On the other hand, we have $K_{\Omega^{j_k}_w,p}(z)\le K_{\Omega,p}(z)$, so that $K_{\Omega^{j_k}_w,p}(z)\rightarrow K_{\Omega,p}(z)$. As the sequence $\{K_{\Omega^j_w,p}(z)\}$ is increasing, we see that $K_{\Omega^j_w,p}(z)\rightarrow K_{\Omega,p}(z)$.
\end{proof}

\begin{proof}[Proof of Theorem \ref{th:completeness_isolated}/$(1)$]
By Corollary \ref{cor:b_5} and \eqref{eq:p-Lower_2},  it suffices to show that $A^\infty(\Omega,w)$ lies dense in $A^2(\Omega,\log K_p)$. Without loss of generality, we assume that $K_{\Omega^j_w,p}(z)>e$. Set 
$$
\varphi_j(z)=\log K_{\Omega^j_w,p}(z),\ \ \ \psi_j(z)=-\frac12 \log \varphi_j(z).
$$
We have
\begin{eqnarray*}
 i\partial\bar{\partial} (\varphi_j+\psi_j) & = &  i\partial\bar{\partial} \varphi_j - \frac{i\partial\bar{\partial} \varphi_j}{2\varphi_j}
  + \frac{i\partial \varphi_j\wedge \bar{\partial}\varphi_j}{2\varphi_j^2}\\
  & \ge & 2 i\partial \psi_j\wedge \bar{\partial}\psi_j
\end{eqnarray*}
on $\Omega^j_w$. Let $\chi:\mathbb R\rightarrow [0,1]$ be a smooth function satisfying $\chi|_{(-\infty,-\log 2]}=1$ and $\chi|_{[0,\infty)}=0$. 
   Given $f\in A^2(\Omega,\log K_p)$ and $0<\varepsilon\ll 1$,  define
   $$
   \eta_{j,\varepsilon}:= f\chi(-2\psi_j-\log\log1/\varepsilon).
   $$
   Since 
   $$
   -2\psi_j-\log\log1/\varepsilon\le 0 \iff K_{\Omega^j_w,p}\le 1/\varepsilon,
      $$
      it follows from \eqref{eq:p-Lower_1} that there exists a constant $C_1>0$ (depending only on $p$ and $R$) such that 
$$
\text{supp}\, \eta_{j,\varepsilon}\subset \{z\in \Omega^j_w: \delta_{j}(z)\ge (C_1\varepsilon)^{1/p}\}.
$$
By (\ref{eq:b-dist_1}) and (\ref{eq:b-dist_2}), we see that  for some $j_\varepsilon\gg 1$
$$
\text{supp}\, \eta_{j,\varepsilon}\subset \{z\in \Omega: \delta(z)\ge (C_1\varepsilon/2)^{1/p}\},\ \ \ \forall\,j\ge j_\varepsilon.
$$
Analogously, as 
$$
-2\psi_j-\log\log1/\varepsilon\le -\log 2 \iff K_{\Omega^j_w,p}\le 1/\sqrt{\varepsilon},
$$
it follows again from \eqref{eq:p-Lower_1}  that 
\begin{eqnarray*}
\{ \eta_{j,\varepsilon}=1 \} & \supset &  \{z\in \Omega^j_w: \delta_{j}(z)\ge \varepsilon^{1/4}/\pi^{1/2}\} \\
& \supset &     \{z\in \Omega: \delta(z)\ge \varepsilon^{1/4}/\pi^{1/2}\}  
\end{eqnarray*}
since $\delta_j\ge \delta$ on $\Omega$. By the Donnelly-Fefferman estimate (cf. \cite{BCh}), there is a solution $u_{j,\varepsilon}$ of the equation
$
\bar{\partial} u=\bar{\partial}\eta_{j,\varepsilon}
$
together with the following estimate
\begin{eqnarray}\label{eq:M_1}
\int_{\Omega^j_w}  |u_{j,\varepsilon}|^2 e^{- \varphi_j} & \le & C_0 \int_{\Omega^j_w} 
            |\bar{\partial}\eta_{j,\varepsilon}|^2_{i\partial\bar{\partial}( \varphi_j+\psi_j)} e^{- \varphi_j} \nonumber\\
            & \le & C_2    \int_{(C_1\varepsilon/2)^{1/p}\le \delta\le \varepsilon^{1/4}/\pi^{1/2}} \frac{|f|^2}{K_{\Omega^j_w,p}}\nonumber\\
            & \le &    2 C_2    \int_{ (C_1\varepsilon/2)^{1/p}\le\delta\le \varepsilon^{1/4}/\pi^{1/2}} \frac{|f|^2}{K_{\Omega,p}}
\end{eqnarray}
whenever $j\ge j_\varepsilon\gg 1$, in view of Lemma \ref{lm:M_2} and the dominated convergence theorem. Here $C_0$ is a universal constant and $C_2$ is a constant depending only on $p,R$.  It follows that 
$$
f_{j,\varepsilon}:=f\chi(-2\psi_j-\log\log1/\varepsilon)-u_{j,\varepsilon}\in \mathcal O(\Omega^j_w)
\subset A^\infty(\Omega,w)
$$
and 
\begin{eqnarray*}
\int_\Omega \frac{|f_{j,\varepsilon}-f|^2}{K_{\Omega,p}} 
            & \le & 2\int_{\delta\le  \varepsilon^{1/4}/\pi^{1/2}} \frac{|f|^2}{K_{\Omega,p}} 
              + 2\int_\Omega  \frac{|u_{j,\varepsilon}|^2}{K_{\Omega,p}} \\
              & \le & (2
              +  4 C_2)    \int_{\delta\le \varepsilon^{1/4}/\pi^{1/2}} \frac{|f|^2}{K_{\Omega,p}}    
\end{eqnarray*}
in view of (\ref{eq:M_1}),  since $K_{\Omega^j_w,p}\le K_{\Omega,p}$ on $\Omega$. The proof is complete.
\end{proof}

Let $\mathbb D^\ast:=\mathbb D\setminus \{0\}$ be the punctured unit disc.  We have the following elementary fact.

\begin{lemma}\label{lm:Ap}
Set $k_p:=\max\{k\in \mathbb Z^+:k < 2/p\}$,   $0<p<2$.  Then
$$
A^p(\mathbb{D}^\ast)=\left\{\sum_{k=1}^{k_p}  c_k z^{-k}+\tilde{f}(z): c_k\in \mathbb C,\ \tilde{f}\in A^p(\mathbb{D})\right\}.
$$
\end{lemma}
                      
\begin{proof}
Given $f\in A^p(\mathbb{D}^\ast)$,    write $f=f_1+f_2$,  where $f_1$ is the principle  part.   Since $f_2$ is holomorphic on $\mathbb{D}$, it follows that $f_1\in A^p(\mathbb{D}^\ast_{1/2})$,  where $\mathbb D^\ast_r=\{z:0<|z|<r\}$.  Moreover,
$$
\int_{\mathbb{D}^\ast_{1/2}} |f_1|^p= \int_{|\zeta|>2} |f_1(1/\zeta)|^p \cdot |\zeta|^{-4}.
$$
Note that $\tilde{f}_1(\zeta):=f_1(1/\zeta)=\sum_{k=1}^\infty c_k \zeta^k$ is an entire function which satisfies 
$$
\int_{|\zeta|>2} |\tilde{f}_1(\zeta)|^p \cdot |\zeta|^{-4}<\infty.
$$
The mean-value inequality implies that  
\begin{eqnarray*}
|\tilde{f}_1(\zeta_0)|^p & \le & \frac1{\pi (|\zeta_0|/2)^2}  \int_{|\zeta-\zeta_0|< |\zeta_0|/2} |\tilde{f}_1|^p \\
              & \le & C_0 |\zeta_0|^{2}   \int_{|\zeta-\zeta_0|< |\zeta_0|/2}  |\tilde{f}_1|^p \cdot |\zeta|^{-4} \\
              & \le  &  C_0 |\zeta_0|^{2}   \int_{|\zeta|>2} |\tilde{f}_1|^p \cdot |\zeta|^{-4}   
\end{eqnarray*}   
whenever $|\zeta_0|>4$.   Thus    
$$
                      |c_k|=\frac1{2\pi} \left| \int_{|\zeta|=r} \frac{\tilde{f}_1(\zeta) }{\zeta^{k+1}} d\zeta\right|\le 
                      C r^{2/p-k}\rightarrow 0\ \ \ (r\rightarrow \infty)
$$
whenever  $k>2/p$, so that
$$
\tilde{f}_1(\zeta)=\sum_{k=1}^{k_p+1}  c_k \zeta^k, \ \ \ \text{i.e.,}\ \ \ 
f_1(z)= \sum_{k=1}^{k_p+1}  c_k z^{-k}.
$$
If $2/p\notin \mathbb Z^+$, then $k_p+1>2/p$, so that $c_{k_p+1}=0$. If $2/p\in \mathbb Z^+$, then $k_p= 2/p-1$ and 
$$
\int_{\mathbb{D}^\ast_{1/2}}|z^{-(k_p+1)}|^p=  \int_{\mathbb{D}^\ast_{1/2}}|z|^{-2}=\infty.
$$
Since
$$
\int_{\mathbb{D}^\ast_{1/2}}|z^{-k}|^p < \infty,\ \ \ k\leq{k_p},
$$                      
we conclude that $c_{k_p+1}=0$.  
\end{proof}

\begin{proof}[Proof of Theorem \ref{th:completeness_isolated}/(2)]
In view of Proposition \ref{prop:localization},  we may assume without loss of generality that $\Omega$ is the punctured unit disc $\mathbb{D}^\ast$.  By Proposition \ref{prop:asyptotic_punctured} we have
$$
K_p(z)\asymp |z|^{-pk_p}
$$
as $z\rightarrow 0$, where $k_p=\max\{k\in\mathbb{Z}^+:\ k<2/p\}$.  For every $f\in A^2(\mathbb{D}^\ast,\log K_p)$ we have
$$
\int_{\mathbb{D}^\ast_{1/2}} |f|^2 \cdot |z|^{pk_p}<\infty, 
$$
so that 
$$
f(z) = c/z+ h(z) ,
$$
where $h\in A^2(\mathbb{D},\log K_p)$, since $pk_p<2$. Thus
$$
K_{2,p}(z) =  \sum_{k=-1}^\infty a_k |z|^{2k}
$$
where  all $a_k>0$. Note that
$$
\left|\frac{\partial K_{2,p}(z)}{\partial z}\right| =|z|^{-3}(a_{-1}+O(|z|^2))
$$
as $z\rightarrow 0$, and 
$$
K_{2,p}(z)\cdot  \frac{\partial^2 K_{2,p}(z)}{\partial z\partial \bar{z}} = a_{-1}^2 |z|^{-6}+ O(|z|^{-4}).
$$
Thus
\begin{eqnarray*}
\frac{\partial^2 \log K_{2,p}(z)}{\partial z\partial \bar{z}} 
                         & = & K_{2,p}(z)^{-2}\cdot \left( K_{2,p}(z)\cdot \frac{\partial^2 K_{2,p}(z)}{\partial z\partial \bar{z}}                                
                          -  \left|\frac{\partial K_{2,p}(z)}{\partial z}\right|^2\right)
\end{eqnarray*}
is uniformly bounded as $z\rightarrow 0$, i.e., $ds^2_p$ is incomplete at $0$.
\end{proof}

\section{Appendix}
In this appendix,  we study the asymptotic behavior of $K_p(z)$ for $\mathbb D^\ast$ and $0<p<2$ as $z\rightarrow 0$.  Let  $k_p$ be as in Lemma \ref{lm:Ap}.  We are going to show the following 

\begin{proposition}\label{prop:asyptotic_punctured}
\begin{equation}\label{eq:asyptotic_punctured_2}
K_{p}(z)=\frac{2-pk_p}{2\pi|z|^{pk_p}}+\frac{4-pk_p}{2\pi}|z|^{2-pk_p}+o(|z|^{2-pk_p}),\ \ \ z\rightarrow0.
\end{equation}
\end{proposition}

\begin{remark}
If $p=2/m$ ($m\in\mathbb{Z}^+$), then $k_p=m-1$, $pk_p=2-2/m$, and \eqref{eq:asyptotic_punctured_2} becomes
\[
K_{2/m}(z)=\frac{1}{m\pi|z|^{2-2/m}}+\frac{m+1}{m\pi}|z|^{2/m}+o(|z|^{2/m}),\ \ \ z\rightarrow0.
\]
\end{remark}

Since $\mathbb D^\ast$ is invariant under rotations and $K_p(z)$ is locally Lipschitz continuous,  it follows from the transformation rule  (cf.   \cite[Proposition 2.7]{CZ}) that  $K_p(z)=\phi_p(|z|)$,  where $\phi_p:(0,1)\rightarrow\mathbb{R}$ is a locally Lipschitz continuous function.  Moreover,  since $\log K_p(z)$ is subharmonic,  it follows from a classical theorem of Hardy that $\log{\phi_p}$  is convex with respect to $\log{r}$.

Given $f\in{A^p}(\mathbb{D}^\ast)$,  write $f(z)=g(z)/z^{k_p}$ for some $g\in{A^p(\mathbb{D})}$ in view of Lemma \ref{lm:Ap}.  Hence
\begin{equation}\label{eq:relation}
\frac{|f(z)|^p}{\int_{\mathbb{D}^\ast}|f|^p}=\frac{|g(z)|^p}{\int_{\mathbb{D}^\ast}|g(w)|^p/|w|^{pk_p}}\frac{1}{|z|^{pk_p}}.
\end{equation}
Set $\varphi_p(z):=pk_p\log|z|$ and define
\[
A^p(\mathbb{D},\varphi_p):=\left\{g\in\mathcal{O}(\mathbb{D}):\ \|g\|_{p,\varphi_p}^p=\int_{\mathbb{D}}|g|^p e^{-\varphi_p}<\infty\right\},
\]
\begin{equation}\label{eq:disc_kernel_weighted}
K_{p,\varphi_p}(z):=\sup\left\{\frac{|g(z)|^p}{\|g\|_{p,\varphi_p}^p}: g\in{A}^p(\mathbb{D}),g\neq0\right\}.
\end{equation}
By \eqref{eq:relation} we have
\begin{equation}\label{eq:punctured_and_weighted}
K_p(z)=\frac{K_{p,\varphi_p}(z)}{|z|^{pk_p}},\ \ \ z\in\mathbb{D}^\ast.
\end{equation}

For any zero-free $h\in{A^2(\mathbb{D})}$, we have $h^{2/p}\in{A^p(\mathbb{D})}$,  so that
\begin{equation}\label{eq:puncture_1}
K_{p,\varphi_p}(z)\geq \frac{|h(z)|^2}{\int_{\mathbb{D}}|h(w)|^2/|w|^{pk_p}}=\frac{|h(z)|^2}{\|h\|_{2,\varphi_p}^2}.
\end{equation}
On the other hand,  a straightforward calculation shows that 
\[h_j(z)=a_jz^j,\ \ \ j=0,1,2,\cdots,
\ \ \ a_j=\left(\frac{2j+2-pk_p}{2\pi}\right)^{1/2},
\]
is a complete orthonormal basis of $A^2(\mathbb D,\varphi_p)$,  so that
\begin{eqnarray}
K_{2,\varphi_p}(w,z) &=& \sum^\infty_{j=0}\frac{2j+2-pk_p}{2\pi}(w\bar{z})^j\nonumber\\
&=& \sum^\infty_{j=0}\frac{j+1}{\pi}(w\bar{z})^j-\frac{pk_p}{2\pi}\sum^\infty_{j=0}(w\bar{z})^j\nonumber\\
&=& \frac{1}{\pi}\frac{1}{(1-w\bar{z})^2}-\frac{pk_p}{2\pi}\frac{1}{1-w\bar{z}}\nonumber\\
&=& \frac{2-pk_p+pk_pw\bar{z}}{2\pi(1-w\bar{z})^2}.\label{eq:disc_2_Bergman}
\end{eqnarray}
In particular,   $K_{2,\varphi_p}(\cdot,z)$ is zero-free when $|z|<\frac{2-pk_p}{pk_p}$,  so that
\begin{equation}\label{eq:punctured_disc_lower}
K_{p,\varphi_p}(z)\geq\frac{|K_{2,\varphi_p}(z)|^2}{\|K_{2,\varphi_p}(\cdot,z)\|_{2,\varphi_p}^2}=K_{2,\varphi_p}(z),\ \ \ |z|<\frac{2-pk_p}{pk_p}.
\end{equation}
This together with \eqref{eq:punctured_and_weighted}, \eqref{eq:puncture_1} and \eqref{eq:disc_2_Bergman} yield
\begin{equation}\label{eq:puncture_2}
K_p(z)\ge \frac{2-pk_p+pk_p|z|^2}{2\pi|z|^{pk_p}(1-|z|^2)^2},\ \ \ |z|<\frac{2-pk_p}{pk_p}.
\end{equation}

Note that $K_{p,\varphi_p}(z)$ can be also defined at $z=0$.  The mean-value inequality gives
\[
\|g\|_{p,\varphi_p}^p=\int^1_0r^{1-pk_p}\int^{2\pi}_0|g(re^{i\theta})|^pd\theta{dr}\geq2\pi|g(0)|^p\int^1_0r^{1-pk_p}dr=\frac{2\pi}{2-pk_p}|g(0)|^p,
\]
with equality if and only if $g$ is a constant. Hence
\[
K_{p,\varphi_p}(0)=\frac{2-pk_p}{2\pi}.
\]

A standard normal family argument yields that given $z\in\mathbb{D}$,  there exists  $g_z\in{A^p(\mathbb{D},\varphi_p)}$ such that $g_z(z)=1$ and $K_{p,\varphi_p}(z)=1/\|g_z\|_{p,\varphi_p}^p$.   The above argument yields that $g_0\equiv 1$.

\begin{lemma}\label{lm:continuity_maximizer}
\begin{enumerate}
\item[
$(1)$] $K_{p,\varphi_p}(z)$ is locally Lipschitz continuous on $\mathbb{D}$.
\item
[$(2)$] $g_z$ converges locally uniformly on $\mathbb D$ to $g_0$ when $z\rightarrow0$.
\end{enumerate}
\end{lemma}

\begin{proof}
(1) Let $S$ be a compact set in $\Omega$ and $z,w\in{S}$. Take $f:=g_z/\|g_z\|_{p,\varphi}$, we have
\[
K_{p,\varphi_p}(z)^{1/p}=f(z)\leq{|f(w)|+C_S|z-w|}\leq{K_{p,\varphi_p}(w)^{1/p}}+C_S|z-w|,
\]
which implies that $K_{p,\varphi_p}^{1/p}$, and hence $K_{p,\varphi_p}$, are locally Lipschitz continuous in $z$.

(2) Take any sequence $\{z_j\}$ with $z_j\rightarrow0$ as $j\rightarrow\infty$.  Then we have
 $$
 \|g_{z_j}\|_{p,\varphi_p}=K_{p,\varphi_p}(z_j)^{-1/p}\rightarrow{K_{p,\varphi_p}(0)^{-1/p}}=\|g_0\|_{p,\varphi_p}. 
 $$
 In particular, $\|g_{z_j}\|_{p,\varphi_p}$ is bounded,  so that $\{g_{z_j}\}$ forms a normal family, and there is a subsequence $\{g_{z_{j_k}}\}$ converging locally uniformly to some $h\in\mathcal{O}(\mathbb{D})$.

We claim that $h=g_0$.  Indeed,  Fatou's lemma gives
\[
\int_{\mathbb{D}}\frac{|h(w)|^p}{|w|^{pk_p}}\leq\liminf_{k\rightarrow\infty}\int_{\mathbb{D}}\frac{|g_{z_{j_k}}(w)|^p}{|w|^{pk_p}}=\frac{1}{K_{p,\varphi_p}(0)}.
\]
Moreover, since
\[
|h(0)-1|=|h(0)-g_{z_{j_k}}(z_{j_k})|\leq|h(0)-h(z_{j_k})|+|h(z_{j_k})-g_{z_{j_k}}(z_{j_k})|\rightarrow0\ \ \ (k\rightarrow\infty),
\]
we have $h(0)=1$.  Thus
\[
K_{p,\varphi_p}(0)\leq\frac{|h(0)|^p}{\int_{\mathbb{D}}|h(w)|^p/|w|^{pk_p}}.
\]
Since $g_0$ is the unique maximizer  of \eqref{eq:disc_kernel_weighted} at $z=0$,  we conclude that  $h=g_0$.  Since the sequence $\{z_j\}$ can be chosen arbitrarily,  we conclude the proof.
\end{proof}

 Lemma \ref{lm:continuity_maximizer} implies that for any $0<\rho<1$,  $g_z$ is zero-free on  $\mathbb{D}_\rho:=\{w: |w|<\rho\}$ for $|z|\ll1$.  In particular, there exists $h_z\in{\mathcal{O}(\mathbb{D}_\rho)}$ such that $h_z^{2/p}=g_z$. 

\begin{lemma}\label{lm:integrating_on_circles}
For any $g\in{A^p(\mathbb{D})}$ and $0<\rho<1$, we have
\[
\int_{\mathbb{D}_\rho}\frac{|g(w)|^p}{|w|^{pk_p}}\leq \rho^{2-pk_p}\int_{\mathbb{D}}\frac{|g(w)|^p}{|w|^{pk_p}}.
\]
\end{lemma}

\begin{proof}
If we set
\[
M(t)=\int^{2\pi}_0|g(te^{i\theta})|^pd\theta,
\]
then
\begin{eqnarray*}
\int_{\mathbb{D}_\rho}\frac{|g(w)|^p}{|w|^{pk_p}} &=& \int^\rho_0t^{1-pk_p}M(t)dt
\ =\  \frac{1}{2-pk_p}\int^\rho_0M(t)d(t^{2-pk_p})\\
&=& \frac{1}{2-pk_p}\int^{\rho^{2-pk_p}}_0M\left(s^{\frac{1}{2-pk_p}}\right)ds,
\end{eqnarray*}
and analogously, 
\[
\int_{\mathbb{D}}\frac{|g(w)|^p}{|w|^{pk_p}}=\frac{1}{2-pk_p}\int^1_0M\left(s^{\frac{1}{2-pk_p}}\right)ds.
\]
Since $M(t)$ is an increasing function, we have
\[
\frac{1}{1-\rho^{2-pk_p}}\int^1_{\rho^{2-pk_p}}M\left(s^{\frac{1}{2-pk_p}}\right)ds\geq\frac{1}{\rho^{2-pk_p}}\int^{\rho^{2-pk_p}}_0M\left(s^{\frac{1}{2-pk_p}}\right)ds,
\]
so that
\begin{eqnarray*}
\int^1_0M\left(s^{\frac{1}{2-pk_p}}\right)ds &=&\int^{\rho^{2-pk_p}}_0M\left(s^{\frac{1}{2-pk_p}}\right)ds+\int^1_{\rho^{2-pk_p}}M\left(s^{\frac{1}{2-pk_p}}\right)ds\\
&\geq& \left(1+\frac{1-\rho^{2-pk_p}}{\rho^{2-pk_p}}\right)\int^{\rho^{2-pk_p}}_0M\left(s^{\frac{1}{2-pk_p}}\right)ds\\
&=&\frac{1}{\rho^{2-pk_p}}\int^{\rho^{2-pk_p}}_0M\left(s^{\frac{1}{2-pk_p}}\right)ds,
\end{eqnarray*}
which completes the proof.
\end{proof}

Lemma \ref{lm:integrating_on_circles} implies that for $|z|\ll1$,  
\begin{eqnarray}
K_{p,\varphi_p}(z) &=& \frac{|g_z(z)|^p}{\int_{\mathbb{D}}|g_z(w)|^p/|w|^{pk_p}}
\ \leq\  \rho^{2-pk_p} \frac{|g_z(z)|^p}{\int_{\mathbb{D}_\rho}|g_z(w)|^p/|w|^{pk_p}}\nonumber\\
&=& \rho^{2-pk_p} \frac{|h_z(z)|^2}{\int_{\mathbb{D}_\rho}|h_z(w)|^2/|w|^{pk_p}}
\ \leq \  \rho^{2-pk_p} K_{\rho,2,\varphi_p}(z),\label{eq:disc_comparison}
\end{eqnarray}
where $K_{\rho,2,\varphi_p}$ is the kernel of $A^2(\mathbb{D}_\rho,\varphi_p)$.   Similar as  \eqref{eq:disc_2_Bergman},  we obtain
\begin{equation}\label{eq:disc_2_Bergman_smaller}
K_{\rho,2,\varphi_p}(z)=\frac{1}{\rho^{2-pk_p}}\sum^\infty_{j=0}\frac{2j+2-pk_p}{2\pi{\rho^{2j}}}|z|^{2j}=\frac{1}{\rho^{2-pk_p}}\frac{2-pk_p+pk_p|z|^2/\rho^2}{2\pi(1-|z|^2/\rho^2)^2}.
\end{equation}
This together with \eqref{eq:disc_comparison} yield
\[
K_{p,\varphi_p}(z)\leq\frac{2-pk_p+pk_p|z|^2/\rho^2}{2\pi(1-|z|^2/\rho^2)^2},\ \ \ |z|\ll 1,
\]
so that
\begin{equation}\label{eq:puncture_3}
K_{p}(z) \leq\frac{2-pk_p+pk_p|z|^2/\rho^2}{2\pi|z|^{pk_p}(1-|z|^2/\rho^2)^2},  \ \ \ |z|\ll 1.
\end{equation}

\begin{proof}[Proof of Proposition \ref{prop:asyptotic_punctured}]
We infer from \eqref{eq:disc_2_Bergman}, \eqref{eq:puncture_2}, \eqref{eq:disc_2_Bergman_smaller} and \eqref{eq:puncture_3} that
\[
\frac{2-pk_p}{2\pi}+\frac{4-pk_p}{2\pi}|z|^2+\cdots\leq{K_{p,\varphi_p}(z)}\leq\frac{2-pk_p}{2\pi}+\frac{4-pk_p}{2\pi\rho^2}|z|^2+\cdots
\]
for $|z|\ll 1$.  For fixed $0<\rho<1$,  we have
\[
\frac{4-pk_p}{2\pi}\leq\liminf_{z\rightarrow0}\frac{K_{p,\varphi_p}(z)-\frac{2-pk_p}{2\pi}}{|z|^2}\leq\limsup_{z\rightarrow0}\frac{K_{p,\varphi_p}(z)-\frac{2-pk_p}{2\pi}}{|z|^2}\leq\frac{4-pk_p}{2\pi\rho^2}.
\]
Letting $\rho\rightarrow1$, we conclude that
\[
K_{p,\varphi_p}(z)=\frac{2-pk_p}{2\pi}+\frac{4-pk_p}{2\pi}|z|^2+o(|z|^2)\ \ \ (z\rightarrow0),
\]
which implies \eqref{eq:asyptotic_punctured_2}.
\end{proof}

We know that $K_{p}(z)=\phi_p(|z|)$, where $\phi_p$ is convex with respect to $\log{r}$. Since $A^p(\mathbb{D}^\ast)\subset{A^p(\mathbb{D})}$, we have
\[
K_{\mathbb{D}^\ast,p}(z)\geq{K_{\mathbb{D},p}(z)}=\frac{1}{\pi(1-|z|^2)^2},\ \ \ \forall\, z\in\mathbb{D},
\]
so that $\phi_p(r)\rightarrow+\infty$ as $r\rightarrow1-$. On the other hand,  Proposition \ref{prop:asyptotic_punctured} implies $\phi_p(r)\rightarrow+\infty$ as $r\rightarrow0+$ for $0<p<2$. Hence for $0<p<2$, there exists some $r_p\in(0,1)$ such that $\phi_p$ is decreasing on $(0,r_p)$ and increasing on $(r_p,1)$. In particular, $K_{p}(z)$ attains the minimum $\phi_p(r_p)$ when $|z|=r_p$.  A natural question is 
\begin{problem}
What are the precise values of $r_p$ and $\phi_p(r_p)$?
\end{problem}

{\bf Acknowledgements.} The authors would like to thank Liangying Jiang for catching a big error in the proof of Theorem \ref{th:HL_2}.

\end{document}